\documentclass[12pt]{amsart}
\usepackage{amssymb,latexsym,amsthm, amsmath, comment, amscd}
\usepackage[all]{xy}
\usepackage{graphicx, fullpage, enumitem}

\ifx\pdfoutput\undefined
\usepackage{hyperref}
\else
\usepackage[pdftex,colorlinks=true,linkcolor=blue,urlcolor=blue]{hyperref}
\fi

\theoremstyle{plain}
\newtheorem{thm}{Theorem}
\newtheorem{main}{Main Theorem}
\newtheorem{lem}{Lemma}
\newtheorem{prop}{Proposition}
\newtheorem*{thm*}{Theorem}

\theoremstyle{definition}
\newtheorem*{ques*}{Question}
\newtheorem*{conv*}{Convention}
\newtheorem*{shoutouts}{Acknowledgements}
\newtheorem{rem}{Remark}
\newtheorem{defin}{Definition}

\newtheorem{cor}{Corollary}

\begin{document}

\title[Logarithmic laws and unique ergodicity]
      {Logarithmic laws and unique ergodicity}
\author{Jon Chaika}
\address{Department of Mathematics\\
         University of Utah\\
         Salt Lake City, UT}         \email{chaika@math.utah.edu}
\author{Rodrigo Trevi\~no}
\address{Brooklyn College, City University of New York \\
         Brooklyn, NY}
\email{rodrigo@trevino.cat}
\date{\today}

\begin{abstract}
We show that Masur's logarithmic law of geodesics in the moduli space of translation surfaces does not imply unique ergodicity of the translation flow, but that a similar law involving the flat systole of a Teichm\"{u}ller geodesic does imply unique ergodicity. It shows that the flat geometry has a better control on ergodic properties of translation flow than hyperbolic geometry.
\end{abstract}
\maketitle
Let $X$ be a compact Riemann surface of genus $g>1$ and $\omega$ a holomorphic 1-form on $X$ and denote by $\Sigma\subset X$ the set of zeros of $\omega$. The pair $(X,\omega)$ defines a \textbf{flat surface} since there exists a maximal atlas, defined by $\omega$, where transition functions on $X\backslash \Sigma$ are of the form $\varphi_{a,b}:z\mapsto  z + c_{a,b}$, for some $c_{a,b}\in\mathbb{C}$. This in turn allows $(X,\omega)$ to have a flat metric at every point of $X\backslash\Sigma$. We will always assume that we are working with flat surfaces of area 1. The set $\Sigma$ consists of the \textbf{singularities} of $(X,\omega)$. Every point in the singularity set $\Sigma$ is not flat: there are charts centered at these points where, in polar coordinates, give them angles of the form $2\pi(k+1)$ for some $k\in\mathbb{N}$ (which is the order of the zero).

A flat surface $(X,\omega)$ is a foliated space: since $\omega$ is holomorphic, the distributions $\ker \Re (\omega)$ and $\ker \Im (\omega)$ define two foliations, the \textbf{vertical and horizontal foliations}, respectively, which are singular at the points of $\Sigma$. The unit-time parametrization of these foliations defines the \textbf{vertical and horizontal flows}, respectively, which preserve the canonical Lebesgue measure $\mu_\omega$ coming from the flat metric. We denote them, respectively, by $\varphi^v_t$ and $\varphi_t^h$. Whenever we refer to the translation flow we will always mean the vertical flow.

A \textbf{saddle connection} $\gamma$ is defined to be a straight segment on $(X,\omega)$ whose endpoints are contained in $\Sigma$. Let $\Gamma = \gamma_1\cup \cdots \cup \gamma_n$ be a union of saddle connections of $(X,\omega)$. Then we denote by $|\Gamma| = \sum_i |\gamma_i|$, where $|\gamma|$ denotes the length, measured with respect to the flat metric on $(X,\omega)$, of $\gamma$. The \textbf{vertical} and \textbf{horizontal components}, respectively, for a finite union of saddle connections $\Gamma = \bigcup_i \gamma_i$ are defined as
$$v(\Gamma) =\sum_i v(\gamma_i) = \sum_i \left|  \int_{\gamma_i} \Im(\omega)   \right|  \hspace{.5in}\mbox{ and }\hspace{.5in} h(\Gamma) =\sum_i h(\gamma_i) = \sum_i \left|  \int_{\gamma_i} \Re(\omega)   \right|.$$

For any flat surface $(X,\omega)$, there exists a (not necessarily unique) collection of saddle connections $\Gamma(\omega)$ with the property that $|\Gamma(\omega)|$ is minimized among all other collections of saddle connections $\Gamma'$ which represent a non-contractible closed-curve. Although the collection $\Gamma(\omega)$ may not be unique, its length $\delta(\omega) := |\Gamma(\omega)|$ is uniquely defined, and it is called the \textbf{systole} of $(X,\omega)$.

Every flat surface $(X,\omega)$ belongs to a \textbf{moduli space} of flat surfaces of genus $g$, $\mathcal{A}_g$. 
 The space $\mathcal{A}_g$ is a non-compact, finite dimensional orbifold and it is stratified by different possible singularity patterns of $\Sigma$. In other words, it is stratified by the number of zeros of $\omega$ and the corresponding orders of the zeros, and different strata are denoted by $\mathcal{A}_g(\bar{k})$, where $\bar{k}$ denotes the singularity pattern associated to the stratum. We denote by $\mathcal{A}_g^{(1)}$ the set of all flat surfaces of genus $g$ and unit area. The space $\mathcal{A}_g^{(1)}$ is exhausted by compact sets of the form 
\begin{equation}
\label{eqn:Mahler}
K_g(\varepsilon) = \{(X,\omega)\in \mathcal{A}_g^{(1)} : X\mbox{ is a Riemann surface of genus $g$ and }\delta(\omega)\geq \varepsilon\}.
\end{equation}

The group $SL(2,\mathbb{R})$ parametrizes a group of deformations of the flat structure of $(X,\omega)$. More precisely, for any $A\in SL(2,\mathbb{R})$, we obtain a new flat surface $A\cdot(X,\omega) = (A \cdot X,A\cdot\omega)$ from $(X,\omega)$ by post-composing the charts of $(X,\omega)$ with $A$. This operation induces a group action of $SL(2,\mathbb{R})$. The moduli space $\mathcal{A}_g^{(1)}$ is equipped with an absolutely continuous, $SL(2,\mathbb{R})$-invariant, probability measure $\mu_g$ \cite{masur1,veech:teich}. The diagonal subgroup $g_t = \mbox{diag}(e^t,e^{-t})$ induces a 1-parameter family of deformations called the \textbf{Teichm\"{u}ller flow}, which also preserves the measure $\mu_g$. Denote by $r_\theta$ an element of the rotation subgroup $SO(2,\mathbb{R})$ and $r_\theta(X,\omega) = (X,\omega_\theta)$. We denote by $\Gamma_t(\omega) := \Gamma(g_t \omega)$ and $\delta_t(\omega) := |\Gamma_t(\omega)| = |\Gamma(g_t \omega)|$. 
Since closed curves can be separating and non-separating, we denote by $\delta^s_t(\omega)$ and $\delta_t^{\not s}(\omega)$, the length of the shortest separating and non-separative curves, respectively, on $g_t(X,\omega)$. As such, we have that $\delta_t(\omega)\in\{\delta_t^s(\omega),\delta_t^{\not s}(\omega)\}$.

One of the most basic questions about translation flows are those concerning the ergodic properties of the flow. The first general results were discovered by Masur which are usually labeled as ``Masur's Criterion''. 
\begin{thm*}[Masur's Criterion \cite{masur1, masur2}]
If the Teichm\"{u}ller orbit $g_t (X,\omega)$ of $(X,\omega)$ does not leave every compact set in moduli space, then the translation flow is uniquely ergodic.
\end{thm*}
Masur's results were in turn used in \cite{KMS} to show that the translation flow for \emph{any} flat surface is uniquely ergodic in almost every direction (a slightly more restricted criterion was proven in \cite{masur1}, which is the one which was used in \cite{KMS}). Masur's criterion relies on the rich interplay between the deformation of a flat surface using the Teichm\"{u}ller deformation and its orbit in the moduli space. Note that by (\ref{eqn:Mahler}), the condition of a surface being ``far away'' in moduli space is directly linked to one of its geometric quantities, namely its systole. More general and quantitative criteria for unique ergodicity have been developed \cite{CheungEskin, rodrigo:area1} in the same spirit as Masur's. In this paper we will focus on the following.

\begin{thm}[\cite{rodrigo:area1}]
\label{thm:UE}
If
\begin{equation}
\label{eqn:divergence}
\int_0^\infty \delta_t^2(\omega)\, dt = +\infty
\end{equation}
then the vertical flow on $(X,\omega)$ is uniquely ergodic.
\end{thm}
In line with the spirit of Masur's criterion, Theorem \ref{thm:UE} says that if the (flat) geometry of a surface evolving under the Teichm\"uller deformation does not degenerate very quickly, then the translation flow is uniquely ergodic. 

In another work, Masur proved \cite{Masur:loglaw} that the behavior of a typical surface in moduli space obeys a certain ``logarithmic law'' \emph{\`{a} la} Sullivan \cite{sullivan}. To state it more precisely, let $\mbox{dist}(\omega,g_t\omega)$ be the Teichm\"uller distance in the moduli space between the surface carrying an Abelian differential $\omega$ and the one carrying $g_t\omega$, which is its orbit under the Teichm\"{u}ller flow. 
\begin{thm*}[Logarithmic Law for Geodesics in Moduli Space \cite{Masur:loglaw}]
For any flat surface $(X,\omega)$, for almost every $\theta\in S^1$,
\begin{equation}
\label{eqn:loglaw}
\limsup_{t\rightarrow \infty}\frac{\mathrm{dist}(\omega_\theta,g_t\omega_\theta)}{\log t} = \frac{1}{2}.
\end{equation}
\end{thm*}
In doing this he showed:
\begin{thm*}\cite{Masur:loglaw}
For any flat surface $(X,\omega)$, for almost every $\theta\in S^1$,
\begin{equation}
\label{eqn:loglaw3}
\limsup_{t\rightarrow\infty}\frac{-\log(\delta_t( \omega_\theta))}{\log t} = \frac{1}{2}.
\end{equation}
\end{thm*}
We note that the theorem as we stated here is not stated explicitly in \cite{Masur:loglaw}, but its proof can be easily retrieved from the proof of (\ref{eqn:loglaw}): he proved the upper bound in Proposition 1.2 and the lower bound in Section 2.

There is a relationship between the Teichm\"uller distance in moduli space $\mathrm{dist}(\omega,g_t\omega)$ and the systole $\delta_t(\omega)$. This relationship in fact involves the lengths of separating and non-separating shortest curves, $\delta_t^s(\omega)$ and $\delta_t^{\not s}(\omega)$, respectively. It is known (see \S \ref{sec:geometry}) that there is a constant $C$, which depends only on the topology of the surface, such that
\begin{equation}
\label{eqn:distanceBound}
\mathrm{dist}(\omega,g_t\omega) \leq   \max\left\{ \frac{1}{2}\log( -\log (\delta^s_t(\omega))), - \log (\delta_t^{\not s}(\omega))\right\} + C.
\end{equation}


Whenever we say that a flat surface $(X,\omega)$ satisfies Masur's logarithmic law, we mean that
\begin{equation}
\label{eqn:loglaw2}
 \limsup_{t\rightarrow \infty}\frac{\mbox{dist}(\omega,g_t\omega)}{\log t} = \frac{1}{2}. 
\end{equation}
Given the criterion for unique ergodicity (\ref{eqn:divergence}), Masur's logarithmic laws (\ref{eqn:loglaw}) and (\ref{eqn:loglaw3}), and the bound (\ref{eqn:distanceBound}) we may wonder whether there is any relationship between logarithmic laws and unique ergodicity. The purpose of this paper is to give some answers to these questions.

To set up our first result, we first define $\phi_\eta(t) := t^{-\frac{1}{2}}(\log t)^{-(\frac{1}{2}+\eta)}$. What would happen if $\mbox{dist}(\omega,g_t\omega) = \phi_\eta(t)$ or $\delta_t(\omega) = \phi_\eta(t)$ for some $\omega$? As such, we have that $\int_1^\infty \phi_0(t)^2\, dt = \infty$ while $\int_1^\infty \phi_\varepsilon(t)^2\, dt < \infty$ for any $\varepsilon > 0$. In either case, both satisfy $\frac{-\log \phi_\eta(t)}{\log t}\rightarrow \frac{1}{2}$ for any $\eta$. As such, it may seem at first that obeying a logarithmic law may not be sufficient for unique ergodicity through the diverging integral (\ref{eqn:divergence}). The first result in this paper is to show that a surface which satisfies a logarithmic law (\ref{eqn:loglaw3}) involving the (flat) systole, that is, involving the quantity $-\log \delta_t(\omega)$ has a uniquely ergodic vertical flow.
\begin{main}
\label{thm:main}
The logarithmic law 
\begin{equation}
\label{eqn:SystLogLaw}
\limsup_{t\rightarrow \infty}\frac{ -\log\, \delta_t(\omega)}{\log t} = \frac{1}{2}
\end{equation}
implies unique ergodicity.

\end{main}
Indeed, we show that if a flat surface $(X,\omega)$ satisfies the logarithmic law (\ref{eqn:SystLogLaw}) then it satisfies the non-integrability condition in (\ref{eqn:divergence}) and thus that the vertical flow is uniquely ergodic. Note that this gives an alternate proof of the theorem of Kerckhoff, Masur and Smillie \cite[Theorem 1]{KMS} that almost every direction on a translation surface has a uniquely ergodic flow. 

However, Masur's logarithmic law (\ref{eqn:loglaw2}) does not imply unique ergodicity. The following result answers a question raised in \cite{rodrigo:area1}.
\begin{main}
\label{thm:negative}
Masur's logarithmic law does not imply unique ergodicity: there exists a flat surface $(X,\omega)$ of genus 2 such that
\begin{equation}
\label{eqn:inequality}
\limsup_{t\rightarrow\infty}\frac{\mathrm{dist}(\omega,g_t\omega)}{\log t} \leq  \frac{1}{2}
\end{equation}
and the vertical flow is not ergodic.
\end{main}
What explains the distinction between the main theorems in light of (\ref{eqn:distanceBound})? The short answer is that given a theorem of Kerckhoff (stated as Theorem \ref{thm:kerckhoff} in \S \ref{sec:geometry}), the distance in moduli space is measured by the extremal length of homotopically non-trivial closed curves on a surface, and not by the flat length of these curves, which is the type of length considered for the systole. More specifically, as a surface degenerates, that is, as the Teichm\"uller orbit of a surface leaves compact sets of moduli space, the flat length becomes a worse reference for the extremal length, and thus for telling distances in moduli space. As such, \textbf{the flat geometry controls unique ergodicity}. To prove the Main Theorem \ref{thm:negative} we devise a construction (following \cite{kat:erg}, \cite{sat:erg} and \cite{veech:strict}) to take advantage of this observation to construct a surface which does not diverge in moduli space too quickly but has systoles which degenerate quickly enough to make (\ref{eqn:divergence}) fail and even to have a non-ergodic translation flow. We do this in Sections \ref{sec:geometry} and \ref{sec:example}.

Our approach to prove the Main Theorem \ref{thm:main}, is to show that if a surface satisfies the logarithm law (\ref{eqn:SystLogLaw}) then there exists a $C_g>0$ depending only on genus so that for a set of $t$ of positive lower density we have that
\begin{equation}
\label{eq:key estimate} 
\frac{-\log\delta_{t}(\omega)}{\log(t)} \leq \frac{1}{2} - C_g.
\end{equation} 
This allows us to show that the criterion for unique ergodicity in \cite{rodrigo:area1} (Theorem \ref{thm:UE} above) is satisfied. To prove (\ref{eq:key estimate}) we assume that there exists an $\epsilon>0$ and a large interval so that 
\begin{equation}\label{eq:clump}\frac{-\log \delta_t(\omega) }{\log(t)}\in \left[\frac 1 2-\epsilon,\frac 1 2 +\epsilon\right]
\end{equation} for all $t$ in this interval. During this long stretch of $t$, many different curves become the shortest curves at different times. We use the technique of combining complexes due to Kerckhoff, Masur and Smillie \cite{KMS} to build subcomplexes of our surface from saddle connections which make up these the simple closed curves which have the shortest length at different times in this stretch of $t$. Our Proposition \ref{prop:baby} shows that, in the presence of (\ref{eq:clump}), curves that become the shortest at later times have to cross the boundary of complexes which we had already considered. This allows us to combine curves which become the shortest with the complexes we considered earlier.  Doing this enough times, we eventually accumulate enough saddle connections to triangulate the surface. However, our argument is effective and in the presence of (\ref{eq:clump}) we do this while all the curves are simultaneously short and so our complex has small area. This is a contradiction and allows us to prove that (\ref{eq:clump}) cannot be satisfied for too long stretches of time (this is the content of Proposition \ref{prop:contr}). Therefore in the presence of (\ref{eqn:SystLogLaw}), for some $t$ in any long enough stretch we must satisfy (\ref{eq:key estimate}). Proposition \ref{prop:gap} then shows that having (\ref{eq:key estimate}) along a set of positive upper density implies (\ref{eqn:divergence}), which gives unique ergodicity.
\begin{rem}
We should point out that although the logarithmic law (\ref{eqn:SystLogLaw}) is sufficient for unique ergodicity, it is not necessary; the Veech dichotomy for Veech surfaces provide counterexamples. Indeed a residual set of directions on a Veech surface do not satisfy Masur's log law and are also uniquely ergodic.
\end{rem}
\begin{rem}
\label{rem:minimal}
Since the present work is concerned with questions of unique ergodicity, we will be working under the assumption at all times that the vertical flow, whose ergodicity properties we are studying, is \textbf{minimal}, i.e., every orbit which is defined for all time is dense in the surface.
\end{rem}
The example in \S \ref{sec:example} shows that there exists a flat surface $(X,\omega)$ so that $\underset{t\rightarrow \infty}{\limsup}\, \frac{-\log(\delta_t)}{\log t}=1$ and the vertical flow is not uniquely ergodic. This motivates us to ask the following question.
\begin{ques*}Is there a flat surface $(X,\omega)$ so that $\underset{t \rightarrow \infty}{\limsup}\, \frac{-\log(\delta_t)}{\log t}<1$ and the vertical flow on $(X,\omega)$ is not uniquely ergodic? 
\end{ques*}
\begin{shoutouts}
J. C. was supported by NSF grants DMS-135500 and DMS-1452762. R. T. was supported by Supported by BSF grant 2010428, ERC starting grant DLGAPS 279893, and NSF Postdoctoral Fellowship DMS-1204008. Both authors thank ICERM, Oberwolfach and CIRM where work was done on this project. We heartily thank Alex Eskin for pointing out the important distinction between the flat geometry versus the hyperbolic one, especially for comparing distances in moduli space. 
\end{shoutouts}
\section{Complexes and interval exchange maps}
Two saddle connections are \textbf{disjoint} if they overlap at most at their endpoints. 
\begin{defin} A \emph{complex} of $(X,\omega)$ is a closed subset of the surface whose boundary is a union of disjoint saddle connections. If a simply connected region is bounded by a triangle made up of saddle connections in the complex, then the region is in the complex.
If a complex $\mathfrak{C}$ on a surface $(X,\omega)$ can be triangulated so that every saddle connection in the triangulation has length at most $\epsilon$ then we say $\mathfrak{C}$ is an $\epsilon$-complex. 
\end{defin}
Let $(X,\omega)$ be a flat surface. We say a subcomplex $\mathfrak{C}$ is an $(A,h)$ -subcomplex if its area is at most $A$ and the boundary of its interior has horizontal component at least $h$.
\begin{defin}
The \emph{level} of a complex  is the number of saddle connections in the complex.  
\end{defin}

The next lemma proven in \cite{flatgames} is the key technical tool of the paper. It allows us to add a short saddle connection to an $\epsilon$-complex to obtain and $\epsilon'$-complex of one level higher where $\epsilon'$ is still small. This technique goes back to \cite{KMS}.
\begin{lem}
\label{lem:combine}
Let $\mathfrak{C}$ be an $\epsilon$-complex and $\gamma\not\subset \mathfrak{C}$, i.e., $\gamma$ is a saddle connection, with length at most $\epsilon$, which intersects the exterior of $\mathfrak{C}$. Then there exists a complex $\mathfrak{C}' = \mathfrak{C} \cup \{\sigma \}$ formed by adding a saddle connection $\sigma$, disjoint from $\mathfrak{C}$, satisfying $|\sigma|\leq 6\epsilon$.
\end{lem}

\begin{rem}
\label{rmk:mbound}
There exists a constant $\mathcal{M}$ which depends only on the stratum to which a flat surface $(X,\omega)$ belongs, such that if $\mathfrak{C}$ is a level $m$ complex, then $m\leq \mathcal{M}$. The constant $\mathcal{M}$ is given by the number of saddle connections needed to triangulate the surface $(X,\omega)$.
\end{rem}

\subsection{Interval Exchange Maps}
Let $(X,\omega)$ be a flat surface and consider the vertical flow $\varphi_t^v:X\rightarrow X$. Let $\beta$ be a union of saddle connections, all of which are transverse to the vertical flow. Define the map $T:\beta\rightarrow \beta$ by $T(x) = \varphi_{r(x)}^v (x)$, where $r(x)$ is the first return time to $\beta$: $r(x) = \min\{t>0:\varphi_t^v(x)\in\beta\}$. The map $T:\beta\rightarrow \beta$ is an \textbf{interval exchange map} since it induces a map of $\beta$ with finitely many discontinuity points $\{\delta_1,\dots ,\delta_{d+1} \}$ (we think of the endpoints of $\beta$ as discontinuity points as well), where $d$ only depends on the topology of $(X,\omega)$, i.e., on $g$ and $\Sigma$. Since the flow is minimal (see Remark \ref{rem:minimal}), the map $T:\beta\rightarrow\beta$ is well defined and it exchanges $d$ intervals.
 There is an upper bound on $d$ that depends only on the topology. Let $\mathcal{D}_{\Gamma}=\{\delta_1,...,\delta_{d_\Gamma}\}$ be the discontinuities of $T_{\Gamma}$ (again, we consider endpoints of saddle connections as discontinuities). These are preimages of $\Sigma$ under the vertical flow. That is, for each $\delta_i$, there exists an $s_i = s(\delta_i)>0$ so that $\varphi^v_{t}(\delta_i) \rightarrow q\in \Sigma$ as $t\rightarrow s_i$. Let
$$H_\Gamma=\{s_i\}_{i=1}^{d_\Gamma}$$
be these times. For every interval $I_j$ we denote by $\partial^- I_j \in \mathcal{D}_\Gamma$ its left endpoint and $\partial^+ I_j \in \mathcal{D}_\Gamma$ its right endpoint. As such, we denote by $s^\pm(I_j) = s_j^\pm \in H_\Gamma$ the corresponding times for the left and right endpoints. 

The following convention related to maps defined by the union of saddle connections will be crucial for the entirety of the paper.
\begin{conv*}
Let $\mathfrak{C}$ be a complex and $\Gamma = \partial \mathfrak{C}$. In this case, the only maps we will consider will be the first-return maps to saddle connections in $\partial\mathfrak{C}$ \emph{which flow into $\mathfrak{C}$}. In other words, when considering a union of saddle connections $\Gamma$ which make up the boundary of a complex $\mathfrak{C}$ we will consider the induced, first-return map $T:\Gamma'\rightarrow \Gamma'$, where $\Gamma'\subset \partial\mathfrak{C}$ is the largest collection of saddle connections in $\Gamma$ which flow into the interior of $\Gamma$ (that is, for each $p \in \Gamma'$, there exists $\epsilon>0$ so that $\phi^v_s(p)\in \mathfrak{C}$ for all $0<s<\epsilon$). This means that if $\mathfrak{C}$ has empty interior, then no map is defined and, if it does, then it is always true that $\Gamma'\varsubsetneq \Gamma$. For any union of saddle connections $\Gamma$, when we write $\gamma\in\Gamma$ we mean the saddle connection $\gamma$ which is contained in $\Gamma$.

Moreover, if $t_\gamma$ is the first return time of $\gamma\in\Gamma'$ to $\Gamma'$, i.e. $\varphi^v_{t_\gamma}(\gamma)\cap \Gamma \neq \varnothing$, and $\varphi^v_s(\gamma)\cap \Sigma = \varnothing$ for all $s\in (0,t_\gamma]$, then we establish the convention that either $s_{\partial_1\gamma}\in H_{\Gamma}$ or $s_{\partial_2\gamma} \in H_{\Gamma}$ is zero, i.e., one endpoint of $\gamma$ contributes zero in $H_{\Gamma}$.
\end{conv*}

\section{Long orbits imply short closed curves}
\begin{prop}
\label{prop:baby} 
Let $\mathfrak{C}$ be a $(A,h)$-complex on a flat surface $(X,\omega)$. Suppose there exists an $N>1$ and a vertical trajectory of length at least $N\frac Ah$ completely contained in $\mathfrak{C}$. Then there exists a closed geodesic which can be shrunk to size at most $\sqrt{2}d^{\frac{3}{2}} N^{-c}\sqrt{A}$ under $g_t$ for some $t\leq \frac{1}{2}\log(\frac{dN^2A}{h^2})$, where $c,d>0$ and only depend on topology.
\end{prop}

We will prove Proposition \ref{prop:baby} at the end of the section. For now, we briefly sketch the ideas which will be involved in the proof and point to the auxiliary lemmas in this section which will help with each part. Lemma \ref{lem:shrinkable} (with key input of Lemma \ref{lem:we weird}) shows that if a transversal to the flow of length $\ell$ has a point on it with a first return time $\frac{K}\ell$, for some large $K$, then there has to be a saddle connection that can be made short. Lemma \ref{lem:spec chain} and the results from \S \ref{subsec:relation} will allow us to get enough such saddle connections to make a simple closed curve which we will then make short through $g_t$. 

Denote by $\Gamma = \partial \mathfrak{C}$ the boundary of the complex $\mathfrak{C}$. For each interval of continuity $I_k$ of the map $T:\Gamma'\rightarrow \Gamma'$ we may consider the minimum of the infimum of return times of points in $I_k$ to $\Gamma$ and the infimum of times at which some point in $I_k$ leaves $\mathfrak{C}$ under the vertical flow. Let $L_{\Gamma}=\{l_i\}_{i=1}^d$ be these times. Note that points in $\Gamma$ either leave the complex immediately under the flow or their return time is the time at which they leave the complex. 

\begin{defin} Let $M>0, C>1$. We say $S\subset \mathbb{R}_+$ is $(M,C)$-\emph{detached} if 
\begin{enumerate}
\item $ S\cap [0,M)\neq \varnothing.$
\item $S\cap (CM,\infty)\neq \varnothing$.
\item $S\cap [M,CM]=\varnothing$.
\end{enumerate}
\end{defin}
\begin{figure}
\begin{center}
\includegraphics[width=3.5in]{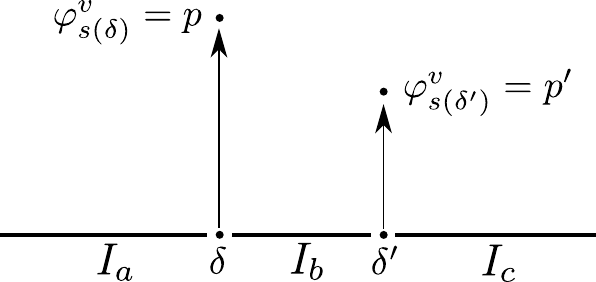}
\end{center}
\caption{The times $s(\delta)$ and $s(\delta')$ are both smaller than the return times of $I_a$ and $I_c$, respectively. Since $|I_b|\leq \frac{\mathrm{Area(subcomplex)}}{\mbox{return time to }I_b}$, and so, if the return time to $I_b$ is much greater than $\max\{s(\delta),s(\delta')\}$, we have a saddle connection between $p$ and $p'$ that is made small under $g_t$.}
\label{fig:return}
\end{figure}
The definition is motivated by our goal of finding saddle connections which can be made short. In order to prove that there are such saddle connections, we consider the return times of different intervals. If two return times of consecutive intervals are vastly different ($(M,C)$-detached for some $M,C$), then we will be able to find a short saddle connection. See Figure \ref{fig:return} for an illustration of the idea: two intervals with quick return times bound an interval with a long return time. As such, we can find a saddle connection.
\begin{lem}
\label{lem:we weird}
Let $\mathfrak{C}$ be an $(A,h)$-complex. Suppose there exists a vertical trajectory of length $N\frac A h>\frac A h$ that stays in the complex. Then there exists $M$ so that $H_\Gamma \cup L_\Gamma$ is $(M,N^c)$ detached for some $c>0$. Moreover, $MN^c$ can be chosen to be less than $N\frac A h$.
\end{lem}
In other words, if the maximal return time of the flow to the boundary of the complex is large enough then there are two consecutive return times with a large ratio.
\begin{proof}
Let $\Gamma = \partial \mathfrak{C}$. By the fact that our flow is injective and the definition of the first return map, $\phi^v_s(x)\neq \phi^v_t(y)$ for all $(x,t)\neq (y,s)$ where $x,y \in \Gamma'$ and $0<s,t<\min L_{\Gamma}$. Thus, $\lambda(\cup_{s=0}^{\min  L_\Gamma}\phi^v_s\Gamma')=h\min L_\Gamma\leq A$ implying that
\begin{equation}
\label{eqn:est1}
\min \{\alpha \in L_{\Gamma}: \alpha  >0\} \leq \frac{A}{h}.
\end{equation}
By the assumption of the lemma we have that $\max L_{\Gamma}\geq N \frac A h$. There is a number $d$, which depends only on the genus of $\omega$ such that $P_\Gamma := |L_\Gamma| + |H_\Gamma| \leq 2d$.
 Consider the interval $\mathcal{I}_\Gamma = [\min \{\alpha\in L_\Gamma:\alpha>0\}  ,   N\frac{A}{h}]\subset \mathbb{R}$ and we consider it partitioned into $r\leq P_\Gamma\leq 2d$ intervals $E_1,\dots, E_r$ by the numbers of $L_{\Gamma}$ and $H_{\Gamma}$ that are less than $N\frac Ah$ . Let $E_i=[a_i,b_i]$ and $\ell_i=\log(b_i)-\log(a_i)$. Note that by (\ref{eqn:est1}) and the fact that the right endpoint of $\mathcal{I}_\Gamma$ is $N\frac{A}{h}$ we have that 
 $$\max \,\ell_i \geq \frac{\log(N\frac A h)-\min \{\alpha\in L_\Gamma:\alpha>0\}}{2d}\geq \frac{\log(N)}{2d}. $$ This implies that $L_\Gamma\cup H_\Gamma$ is $(M,N^c)$-detached with $c = \frac{1}{2d}$. Moreover, by construction, $MN^c\leq N\frac{A}{h}$.
\end{proof}
 For the remainder of this section, we fix an $M$ whose existence is guaranteed by Lemma 2 and set $C = N^c.$ 
\begin{defin} If $L_{\Gamma}$ is $(M,C)$ detached we say $I_j$ is \emph{early} if $l_j<M$. Otherwise it is \emph{late}. 
We say a discontinuity $\delta_i$ is \emph{conflicted} if one of $I_i$, $I_{i+1}$ is early and the other is late.
\end{defin}
\begin{lem}
\label{lem:special early} If $\delta_i$ is conflicted then $s_i<M$.
\end{lem}
\begin{proof}Let $I_a$ be the early interval and $I_b$ be the adjacent late interval to $\delta_i$. They travel together until they hit a discontinuity and since one is early and the other is late they have to separate before $M$. So $s_i<M$.
\end{proof}
\begin{lem}
\label{lem:shrinkable}If $\delta_i$ is conflicted then $\varphi_{s(\delta_i)}(\delta_i)^{ v}$ is an endpoint of a saddle connection with a vertical component at most $M$ and horizontal component at most $\frac A {MC}$.
\end{lem}
\begin{proof} Let $I_a$ be the early interval and $I_b$ be the late interval. By Lemma \ref{lem:special early} $s_i<M$. Let $J$ be the maximal connected set of late intervals that $I_b$ belongs to. So  $\delta_i$ is one endpoint of $J$ and $\delta_k$ is the other. Note $\delta_k$ is also conflicted. Because $\delta_k$ is conflicted $\varphi^v_s(\delta_k)$ hits a singularity for some $ 0\leq s<M$. Notice that because the flow is injective it sweeps out an area at most the subcomplex before first return. Flowing from a horizontal segment of length $s$ for times $t$ sweeps out an area of $st$. Since the flow of a point on $J$ stays in the complex from $0$ to $MC$, its horizontal component is at most $\frac{A}{CM}$. 

Now we can either form a saddle connection between $\varphi_{s(\delta_i)}^{ v}(\delta_i)$ and  $\varphi_{s(\delta_k)}^{ v}(\delta_k)$ satisfying the lemma or there exists a discontinuity $\delta_p$ between $\delta_i$ and $\delta_k$ that hits a singularity before $M$ so that we can form a saddle connection between $\varphi_{s(\delta_i)}^{ v}(\delta_i)$ and $\varphi_{s(\delta_p)}^{ v}(\delta_p)$. In particular $\delta_p$ could be the discontinuity in $J$ closest to $\delta_i$ which hits a singularity in time at most $M$.
\end{proof}
\begin{lem}\label{lem:spec chain} If $\delta_i$ and $\delta_j$ are conflicted discontinuities with every interval between them late then there
is a chain of saddle connections
\begin{itemize}
\item connecting $\varphi_{s(\delta_i)}^{ v}(\delta_i)$ to $\varphi_{s(\delta_j)}^{ v}(\delta_j)$
\item their union is a geodesic segment on the flat surface and
\item so that the sum of the vertical components of these saddle connections is at most $dM$ and the sum of their horizontal components is at most $\frac A {CM}$.
\end{itemize}
\end{lem}
The proof uses the following well known fact (it follows from \cite[Theorem 8.1]{strebel}).
\begin{prop}[Strebel] 
\label{prop:strebel}
A geodesic on a flat surface has the following form. It is made of straight lines connecting singularities of the flat surface. At each singularity the consecutive line segments make an angle of at least $\pi$.
\end{prop}

\begin{proof}[Proof of Lemma \ref{lem:spec chain}] We assume $i<j$ and let $J$ be the interval between them. As in Lemma \ref{lem:shrinkable}, $|J|\leq \frac A{MC}$. Consider $\mathcal{D}'$ the set of  discontinuities between $\delta_i$ and $\delta_j$ that are flowed into before $\max\{s_i,s_j\}<M$. If this is empty the 
proof of the Lemma \ref{lem:shrinkable} implies this lemma. Otherwise let $p_1$ be the discontinuity immediately to the right of $\delta_i$. The proof of the Lemma \ref{lem:shrinkable} gives a short saddle connection between the singularity $\delta_i$ will hit and the one  $p_1$ will hit. Let $J_1$ be the subinterval of $J$ between $\delta_j$ and $p_1$. Now either the flow is continuous on $J_1$ until both hit singularities or there exists $p_2$ a discontinuity closest to $p_1$ that separates the segments before they hit the two singularities. Now if the angle between the saddle connections connecting $\delta_i$ to $p_1$ and $p_1$ to $p_2$ is at least $\pi$ then this is a geodesic segment (Proposition \ref{prop:strebel}). Otherwise there is a shorter geodesic connecting them. Indeed there is a straight line connecting $\delta_i$ and $p_2$: the horizontal component that the saddle connection from $p_1$ to $p_2$ is the length of a horizontal line segment that travels continuously. The line from $\delta_i$ towards $p_2$ hits a vertical translate of this segment and then follows in the flat region it sweeps out until it hits $p_2$. Repeat this process. Inductively each saddle connection has vertical component at most $M$. Their total horizontal component is $|J|$ and there are at most $d$ discontinuities.
 \end{proof}

\subsection{An equivalence relation on discontinuities and a graph to build closed curves}
\label{subsec:relation}
Define a relation on $\mathcal{D}_{\Gamma}$ by $p\sim q$ if $T(p^{\pm})=T(q^{\mp})$ and consider the equivalence relation $\sim$ it generates. Define a graph as follows. The vertices of the graphs are equivalence classes of elements of $\mathcal{D}_\Gamma$ at least one of whose elements are conflicted. The vertices $[p]$, $[q]$ are connected by an edge if there exists $\delta \in [p]$ and $\delta'\in [q]$ which are conflicted and have no conflicted discontinuities between them. Note that if two vertices are connected by such an edge, Lemma \ref{lem:spec chain} says a certain type of a chain of saddle connections exists between $\varphi_{s(\delta)}^{ v}(\delta)$ and $\varphi_{s(\delta')}^{ v}(\delta')$. We allow for edges to connect a vertex to itself and for there to be multiple edges between a pair of vertices.
\begin{rem}
\label{rem:equiv rel} 
Discontinuities of the IET arise as pre-images of cone points and our equivalence identifies discontinuities that come from the same cone point.
\end{rem}
\begin{lem}Every vertex in our graph has valence at least 2.
\end{lem}

\begin{proof}Let $\delta$ be a conflicted discontinuity. One side is early and the other side is late. Follow the identifications from the relation $\sim$ for the side that is early until it becomes late again. Since the relation forms a closed cycle this has to happen. When it does we have another conflicted discontinuity for $T$. It has to be different because we must become late again at least one step before repeating the already given identification. So all vertices have valence at least 2.
\end{proof}
\begin{lem}
\label{lem:cycle}
A finite graph so that each vertex has valence at least 2 contains a cycle.
\end{lem}

\begin{lem}\label{lem:non trivial} Consider the collection of saddle connections corresponding to a cycle in our graph. They form a closed curve that is not homotopically trivial.
\end{lem}
\begin{proof}Consider a cycle in our graph (with edges) $\xi_1,...,\xi_n$ and let $p_1,...,p_n,p_{n+1}=p_1$ the vertices (which are zeros of $\omega$). For each $\xi_i$ its endpoints $p_{i},p_{i+1}$ correspond to discontinuities of $T$, $\delta_{k^1_i}$, $\delta_{k^2_{i+1}}$ and discontinuities of $T^{-1}$, $\delta_{\ell^1_i}', \delta_{\ell^2_{i+1}}$. Now consider the edges $\xi_{i-1}$, $\xi_{i}$ which meet at $p_i$. Because of our equivalence relation, these saddle connections meet at the same cone point (Remark \ref{rem:equiv rel}).
 If the angles between $\xi_{i-1}$ and $\xi_i$ at $p$ are at least $\pi$, then our cycle is a geodesic which is not homotopically trivial. We now examine the other case. 
 If the angle between $\xi_{i-1}$ and $\xi_i$ at $p_i$ is less than $\pi$ we will use that either $k^1_{i-1}=k^2_i$ or $\ell^1_{i-1}=\ell^2_{i}$. Now if Re$(\omega)\geq 0$ as $\xi_{i-1}$ went towards $p_i$ we will also have Re$(\omega)>0$ as $\xi_i$ goes away from $p_i$ (by the previous sentence). 
  Similarly for Re$(\omega)<0$. The geodesic in the fixed endpoint homotopy class of $\xi_{i-1}\cup \xi_i$ needs to preserve  the total increase in Re$(\omega)$ from the start of $\xi_{i-1}$ to the end of $\xi_i$. We iterate this, and end up with a geodesic segment made up of saddle connections that are separated from each other by angle at least $\pi$. Each saddle connection in our geodesic has the same sign of  change in Re$(\omega)$ as  unions of saddle connection giving edges in our graph.  Since this is a closed geodesic and one of the segments has positive length, it is not homotopically trivial.
\end{proof}

\begin{proof}[Proof of Proposition \ref{prop:baby}] 
  Suppose there exists a trajectory starting in $\mathfrak{C}$ of length $N\frac{A}{h}>\frac{A}{h}$ which is contained in $\mathfrak{C}$. By Lemma \ref{lem:we weird} the set $H_\Gamma\cup L_\Gamma$ is $(M,N^c)$-detached for some $M$ and $c$ satisfying $MN^c < N\frac{A}{h}$. Lemmas \ref{lem:special early} - \ref{lem:cycle} imply the existence of a cycle, every edge of which represents a geodesic isometric to a chain of saddle connections with horizontal component at most $\frac{A}{MN^c}$ and vertical component at most $dM$ (Lemma \ref{lem:spec chain}). So we have a path with total holonomy at most $(d\frac{A}{MN^c},d^2M)$. Choose a geodesic in the homotopy class of this path which is nontrivial by Lemma \ref{lem:non trivial}.

Its holonomy will be smaller and it is a closed curve. So we have a closed curve with holonomy $(x,y)$. Its length can be shrunk to $\sqrt{2}\sqrt{xy}$ by the matrix 
$\begin{pmatrix} \frac{\sqrt{y}}{\sqrt{x}} & 0\\
0 & \frac{\sqrt{x}}{\sqrt{y}}
\end{pmatrix}$ in time less than 
$$\frac{1}{2}\log \left| \cot\left(d\frac{A}{CM},d^2M\right)\right| = \frac{1}{2}\log\left(\frac{dM^2N^c}{A}\right),$$
which, combined with the fact that  $M, MN^c \leq N\frac{A}{h}$ (Lemma \ref{lem:we weird}), yields the result.
\end{proof}

\section{Induction setup}
The main result of this section, Proposition \ref{prop:induct}, shows that closed curves which will be shortest at some point in the near future cannot be contained in a subcomplex with small perimeter and not too small area. Moreover, it proves an effective version of this statement.
\begin{prop}
\label{prop:induct} 
Let $(X,\omega)$ be a flat surface for which there exists $t_0$ and $\epsilon> 0$ so that for all $t>t_0$ we have that $\delta_t(\omega)>t^{-\frac 1 2 -\epsilon}$. Given any $C>0$, let $\rho = \frac{2+C}{c} + 3C +10$ (where $c$ is the constant from Proposition \ref{prop:baby}). There exists a $p$ such that if $s>p$  and we have 
\begin{itemize}
\item a complex $\mathfrak{C}$ of $g_s\omega$ built from saddle connections whose length on $g_s\omega$ is between $s^{-\frac 1 2 -\epsilon}$ and $s^{-\frac 1 2 +C\epsilon}$ and so that the complement of $\mathfrak{C}$ is not homotopically trivial and
\item  $\delta_\ell(\omega)<\ell^{-\frac 1 2 +\epsilon}$ for all $\ell \in [s,s+\log(s^{2\rho\epsilon})]$
\end{itemize}
 then $\Gamma_{s+\log(s^{2\rho \epsilon})}(\omega)\not\subset\mathfrak{C}$.
\end{prop}
We first remark that although the shortest curve on $g_{s+\log(s^{2\rho \epsilon})}\omega$ may not be unique, the statement is true for any such curve.

Secondly, we remark that the proposition above is trivial if $\mathfrak{C}$ has no interior, so we assume for the remainder of this section that $\mathfrak{C}$ has interior. Since systoles are not homotopically trivial, we also assume that the interior of the complex is not homotopically trivial.

The proof of this proposition has 3 parts. First, Lemmas \ref{lem:area} and \ref{lem:hor} use the assumptions of the proposition to obtain some bounds on the subcomplex's area and perimeter so we have the assumptions of Proposition \ref{prop:baby}. Second, we relate long (at least compared to $\sqrt{\text{area}}$) simple closed curves contained in a subcomplex to long vertical trajectories contained in a subcomplex (\S \ref{sec:shadow}). The last part is showing that (\ref{eqn:SystLogLaw}) implies that somewhat distant future systoles give long enough vertical trajectories (Lemma \ref{lem:time big}) to apply Proposition \ref{prop:baby} to show the vertical trajectory given in the second part would have to leave the subcomplex, thus completing the proof of Proposition \ref{prop:induct}.
\begin{lem}
\label{lem:area}
Let $(X,\omega)$ be a flat surface so that there exists $t_0$ and $\epsilon\geq 0$ so that for all $t>t_0$ we have that $\delta_t(\omega)>t^{-\frac 1 2 -\epsilon}$. Suppose there exists a $C$ such that for
 $s>t_0$ we have $\mathfrak{C}$, a level $m$ subcomplex  of $g_s\omega$ built from saddle connections whose length on $g_s\omega$ is less than $s^{-\frac 1 2 +C\epsilon}$. Then the area of the subcomplex is at most $m^2s^{-1 +2C\epsilon}$.
\end{lem}
\begin{proof} It is a polygon with total perimeter $ms^{-\frac 1 2 +C\epsilon}$.
\end{proof} 
\begin{lem}
\label{lem:hor}
Let $(X,\omega)$ be a flat surface so that there exists $t_0$ and $\epsilon\geq 0$ so that for all $t>t_0$ we have that $\delta_t(\omega)>t^{-\frac 1 2 -\epsilon}$. Suppose there exists a $C$ such that for
 $s>t_0$ we have $\mathfrak{C}$, a level $m$ subcomplex  of $g_s\omega$ built from saddle connections whose length on $g_s\omega$ is between $s^{-\frac 1 2 -\epsilon}$ and $s^{-\frac 1 2 +C\epsilon}$  and so that $\mathfrak{C}^c$ is not homotopically trivial. Then
on $g_{s+\log(ms^{(C+3)\epsilon})}\omega$ the sum of the horizontal components of the saddle connections that make up the boundary $\partial\mathfrak{C}$ is at least $s^{-\frac 1 2 -3\epsilon}$.
\end{lem}
\begin{proof}By our assumptions the boundary curve has vertical holonomy at most $ms^{-\frac 1 2 +C\epsilon}$. The boundary contains a simple closed curve that  is not homotopically trivial because by assumption the complement of $\mathfrak{C}$ is not homotopically trivial and by the remark after the statement of the proposition we are assuming the interior of $\mathfrak{C}$ is not homotopically trivial. That is, a homotopy of the boundary to a point would also homotope all of the curves in  either the interior or the complement of the complex to that point. 
So by our assumption on the largeness of future systoles it cannot become shorter that $s^{-\frac{1}{2} -\epsilon}$. 
At the time listed the vertical component will be at most 
$$ \frac{1} {ms^{(C+3)\epsilon}}ms^{-\frac 1 2 +C\epsilon} = s^{-\frac 1 2 -3\epsilon}.$$ 
meaning the horizontal component must be larger and so the lemma follows.
\end{proof}
\subsection{$r$-shadowing}\label{sec:shadow}
We will complete the proof of Proposition \ref{prop:induct} by appealing to Proposition \ref{prop:baby}. To do this we first need a lemma to relate a closed geodesic not leaving a subcomplex to a vertical trajectory not leaving a subcomplex. We do this in this section.
\begin{defin}
\label{def:shadow}
Let $\gamma$ be a saddle connection and $a\in \gamma$. We say that the vertical trajectory of $a$ \emph{r-shadows} $\gamma$ if there exists a $\sigma\in\mathbb{R}$ such that for all $\tau\in I_r :=[0,r]$ (if $r>0$; $\tau\in I_r:= [r,0]$ if $r<0$) we have that $\varphi^h_{\sigma\tau}\circ \varphi_\tau^v(a) \in\gamma$ and that the set
$$T(a,\gamma,r):= \bigcup_{\tau\in I_r}\bigcup_{s\in[0,\sigma r]}\varphi_s^h\circ\varphi^v_\tau(a)$$
is a flat triangle. We call $T(a,\gamma,r)$ a \emph{shadowing triangle}.
\end{defin}
See Figure \ref{fig:shadow}, left, for an illustration of this definition.
\begin{figure}[t]
\begin{center}
\includegraphics[width=5.5in]{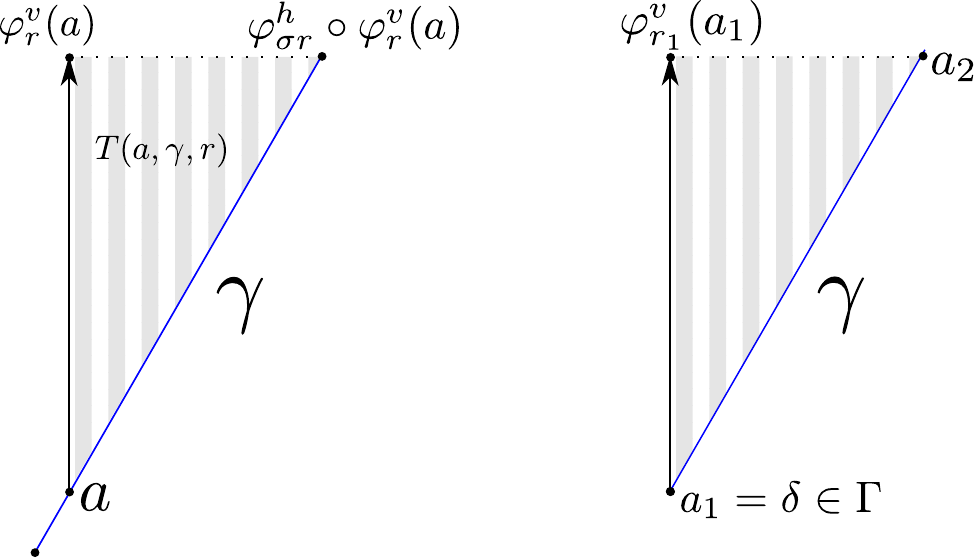}
\end{center}
\caption{On the left, an illustration of the definition of shadowing. On the right, the first step in the recursive procedure used in the proof of Lemma \ref{lem:sys to traj}.}
\label{fig:shadow}
\end{figure}
\begin{lem}\label{lem:sys to traj} 
Let $\Gamma$ be a union of saddle connections that form a closed loop on a flat surface $(X,\omega)$ such that $|g_t\Gamma| = \delta_t(\omega)$ for some $t>0$. Then there exists a $D$, depending only on the stratum to which $(X,\omega)$ belongs, so that there is a point $a\in \Gamma$ such that its vertical trajectory $\frac{v(\Gamma)}{D}$-shadows some $\gamma\in\Gamma$.
\end{lem}
\begin{proof}
Let $a_1 \in\Gamma$ be a singularity, i.e., $a_1\in\partial\gamma$ for some $\gamma\in\Gamma$. Let $r_1$ be the largest number so that $a_1$ $r_1$-shadows $\gamma$ and $\sigma_1$ is the number from Definition \ref{def:shadow} associated to $a_1$ and $\gamma$. Starting with $i=1$, either $a_{i} \in \Sigma$ or not. If it is not, then $a_{i+1}$ is defined as
$$a_{i+1} = \lim_{t\rightarrow r_i} \varphi_{\sigma_it}^h\circ\varphi_t^v(a_i),$$
i.e., $a_{i+1}$ is obtained as a limit by shadowing the same saddle connection $\gamma$ which was shadowed by $a_i$. If $a_i\in\Sigma$, since $\Gamma$ is a closed loop, there is exactly one other edge $\gamma'\in\Gamma$ which contains $a_i$. In this case, we define $a_{i+1}$ to be 
$$a_{i+1} = \lim_{t\rightarrow r_i} \varphi_{\sigma_it}^h\circ\varphi_t^v(a_i),$$
where $r_i$ and $\sigma_i$ are the appropriate quantities obtained from shadowing $\gamma'$. We proceed recursively, given $a_i,\sigma_i$ and $r_i$, to obtain $a_{i+1}, \sigma_{i+1}$ and $r_{i+1}$. Since $\Gamma$ is a closed loop, after finitely many iterations of this process we shadow every edge in $\Gamma$ and return to $ a_1$. Let $k$ be the number of these iterations. 

We claim that $k< 4d$, where $d$ is the number of cone points of $\omega$. First observe that once we have defined $a_d$ there exists a simple closed curve on the flat surface made up of saddle connections where the sum of the (unsigned) vertical and horizontal holonomies is at most the sum of the vertical and horizontal holonomies of the saddle connections in $\gamma$ from $a_1$ to $a_d$. This happens for the following reason: for each shadowing triangle constructed when we go from $a_i$ to $a_{i+1}$, the geodesic segment connecting $a_i$ to $a_{i+1}$ is homotopic, relative to $\{a_i,a_{i+1}\}$, to the union of at most two geodesic segments, one connecting $a_i$ to some $\delta_i\in\Sigma$ and another connecting $\delta_i$ to $a_{i+1}$ (if $a_{i+1}\in\Sigma$ there is only the one geodesic segment connecting $a_i$ to $a_{i+1}$). Moreover, for each $i$, the vertical and horizontal holonomies of these two homotopic curves are the same. Doing this $d$ times we involve the $d+1$ singularities $a_1 = \delta_0,\delta_1,\dots, \delta_{d}$ in the different shadowing triangles. Thus the same singularity must occur at least twice and we therefore have a simple closed curve. Now, if $k\geq 4d$, then one of these closed curves has that both its horizontal and vertical holonomies at most half of that of $\Gamma$. So $\Gamma$ is never the systole contradicting our assumption.

Let $\xi_1,\dots, \xi_k$ be the lengths of different segments of shadowing trajectories which emanate from the points $a_i\in\Gamma$. Then $\sum \xi_i=v(\Gamma) < |\Gamma|$. Thus the largest is at least $\frac {v(\Gamma)}{k}\geq \frac{v(\Gamma)}{4d}$.
\end{proof}

We wish to prove Proposition 3 by invoking Proposition 1. To do this we need the shadowing trajectory to be inside $\mathfrak{C}$. This requires some additional work.

\textbf{Sublemma:} Consider the flat triangles formed in the definition of shadowing. Under the assumptions of Lemma \ref{lem:sys to traj} and the procedure in its proof, any point in such a triangle is a point in at most two such triangles. 

\begin{proof}Consider $g_t\Gamma$ in $g_t\omega$. Consider a point that is in at least two such triangles. We follow the shadowing procedure from Lemma \ref{lem:sys to traj}. For any two consecutive triangles it is contained in, one of these shadowing trajectories shadows the point first. One can link this point back to itself  by a connected subset of $\Gamma$ union two horizontal trajectories, giving an element of a nontrivial homotopy class.

 If the point is in at least three shadowing triangles, one of these has that the sums of vertical and horizontal trajectories is less than half of the sum of vertical and horizontal holonomies of $\Gamma$. Since the length of $\Gamma$ is at least $\frac 1 {\sqrt{2}}$ times the sum of its vertical and horizontal holonomies, there is  closed geodesic that is shorter.  This contradicts that $\Gamma$ is a systole at time $t$.
\end{proof}

\begin{lem}\label{lem:sys to traj inside} 
Let $\Gamma$ be a union of saddle connections contained in $\mathfrak{C}$ which is the systole of $g_t\omega$ for some $t>0$. Then there exists a $D$, depending only on the stratum to which $(X,\omega)$ belongs, so that there is a point $a\in \Gamma$ such that its vertical trajectory $\frac{|\Gamma|}{D}-2v(\partial \mathfrak{C})$-shadows some $\gamma\in\Gamma$ and stays in $\mathfrak{C}$ while it shadows $\gamma$.
\end{lem}

\begin{proof}
This is similar to the previous proof with one small variant. First, there are two ways a shadowing trajectory can leave the complex. 1) It can cross a saddle connection in $\partial \mathfrak{C}$ that is not the geodesic it is shadowing or 2) it can be shadowing a saddle connection in $\partial \mathfrak{C} \cap \Gamma$ and the shadowing trajectory immediately begins flowing  outside of the complex. If either of these happens then on each horizontal line given by a shadowing trajectory while it is outside the complex there is a point in $\partial \mathfrak{C}$. The sublemma establishes that each point in $\partial \mathfrak{C}$ can be in at most two shadowing triangles. So in a shadowing trajectory one can travel for at most $2v(\partial \mathfrak{C})$ outside of the complex and the lemma follows as in Lemma \ref{lem:sys to traj}.
\end{proof}
The next lemma bounds the vertical component of $\Gamma_{s+\log(s^{2\rho \epsilon})}$ on $g_s\omega$ from below. This will allow us to combine Lemma \ref{lem:sys to traj inside} and Proposition \ref{prop:baby} to prove Proposition \ref{prop:induct}. 
 Some notation: if $\Gamma$ denotes a curve on  $\omega$, let $g_\ell\Gamma$ denote the image of this curve under $g_\ell$ on $g_\ell\omega$.
\begin{lem}
\label{lem:time big}
Let $(X,\omega)$ be a flat surface and suppose there exists $t_0$ and $\epsilon\geq 0$ so that for all $t>t_0$ we have that ${\delta_t (\omega)>t^{-\frac 1 2 -\epsilon}}$. For any $\rho > 0$ there exists $s_0$ so that if for some $s>s_0$ we have that
\begin{itemize}
\item $\delta_\ell(\omega)<\ell^{-\frac 1 2 +\epsilon}$ for all $\ell \in [s,s+\log(s^{2\rho \epsilon})]$ and
\item $\Gamma$ is a curve on $(X,\omega)$ with the property that 
${|g_{s+\log(s^{2\rho\epsilon})}\Gamma| = \delta_{s+\log(s^{2\rho\epsilon})}(\omega)}$
\end{itemize} 
then the vertical component of $g_s\Gamma$ on $(X,g_s\omega)$ is at least $s^{-\frac 1 2+(\rho-4)\epsilon}$.
\end{lem}
\begin{proof} Observe that if $\rho<3$ then $(\rho-4)\epsilon<-\epsilon$, thus the assumption implies the conclusion for $s_0=t_0$. So we assume $\rho\geq 3$. 
Choose $s_0$ large enough so that 
$$\frac 1 2 (s+\log(s^{2\rho\epsilon}))^{-\frac 1 2-\epsilon}>s^{-\frac 1 2 -1.5\epsilon} \hspace{.2in}\text{ and } \hspace{.2in}e^{\log(s^{\rho \epsilon})}(s+\log(s^{\rho \epsilon}))^{-\frac 1 2 -\epsilon}>(s+\log(s^{2\rho \epsilon}))^{-\frac1 2 +\epsilon}$$ for all $s>s_0$. Such an $s_0$ clearly exists for the first inequality. For the second inequality, we use that  $\underset{s \to \infty}{\lim}\,\frac{s}{s+\log(s^{\rho \epsilon})}=1$ and so for each $\rho\geq 3>2$ there exists $s_0$ so that
$$ s^{\rho \epsilon}(s+\log(s^{\rho \epsilon}))^{-\frac 1 2 -\epsilon}>(s+\log(s^{\rho \epsilon})^{-\frac 1 2 +\epsilon}>(s+\log(s^{2\rho \epsilon})^{\frac 1 2 +\epsilon}$$
for all $s>s_0$. From now on we assume that $s>\max\{t_0,s_0\}$.

First, observe that since
$$e^{\log(s^{\rho \epsilon})}(s+\log(s^{\rho \epsilon}))^{-\frac 1 2 -\epsilon}>(s+\log(s^{2\rho \epsilon}))^{-\frac1 2 +\epsilon}$$
by our assumptions on the length 
of the systoles we have that the length of $\Gamma$ is decreasing at time $s+\log(s^{\rho \epsilon})$. This implies that
$$v(g_{s+\log(s^{\rho\epsilon})}\Gamma)>\frac 1 {\sqrt{2}}(s+\log(s^{\rho\epsilon}))^{-\frac 1 2 -\epsilon}>s^{-\frac 1 2 -1.5\epsilon}.$$
So $v(g_s\Gamma)>s^{\rho\epsilon}s^{-\frac 1 2 -1.5\epsilon}$ implying the lemma. 
\end{proof}

\subsection{Proof of Proposition \ref{prop:induct}}
To prove Proposition \ref{prop:induct} we will show that, given some large $\rho>0$, if $\Gamma_{s+\log(s^{2\rho \epsilon})}$ is contained in $\mathfrak{C}$ for all $t\in(s,s +\log(s^{\rho\epsilon}))$ then by Proposition \ref{prop:baby} the assumptions of Proposition \ref{prop:induct} are violated. We use the previous results of this section to translate the assumption of Proposition \ref{prop:induct} into the language of Proposition \ref{prop:baby}. The lower bound (\ref{eq:together}) bounds the length of a vertical trajectory contained in $\mathfrak{C}$ in terms of the area and horizontal component of the boundary of $\mathfrak{C}$. The bounds (\ref{eqn:shrinking}) and (\ref{eq:time}) use Proposition \ref{prop:baby} to complete the proof.

Given $C>0$, let $\rho = \frac{2+C}{c}  + 3C + 10$. 
Let $p_* = p_*(\rho,\epsilon)$ be the number from proof of Lemma \ref{lem:time big}, i.e.,   for all $s>p_*$ we have  
$$(s+\log(s^{2\rho\epsilon}))^{-\frac 1 2-\epsilon}>s^{-\frac 1 2 -2\epsilon} \hspace{.2in}\text{ and } \hspace{.2in}e^{\log(s^{\rho \epsilon})}(s+\log(s^{\rho \epsilon}))^{-\frac 1 2 -\epsilon}>(s+\log(s^{2\rho \epsilon}))^{-\frac1 2 +\epsilon}.$$
 Let $p^*$ be big enough so that for any $s>p^*$ we have 

\begin{equation}
\label{eqn:pBnd}
2^{\frac{1}{2}+c} d^{3/2} D^{c}\mathcal{M}^{3c+1} s^{-\frac{1}{2} - 2\epsilon} < (s+\log(s^{2\rho\epsilon}))^{-\frac{1}{2} - \epsilon} \, \text{ and }\mathcal{M}< s^{(\rho-C-4)\epsilon}/8D,
\end{equation}
where $d,D$ and $\mathcal{M}$ are the topological constants from Proposition \ref{prop:baby}, Lemma \ref{lem:sys to traj inside} and Remark \ref{rmk:mbound}, respectively. Let $p = \max\{p_*,p^*,t_0\}$ and we assume from now on that $s>p$.

Let $\mathfrak{C}$ be our complex and $\Gamma$ be a union of saddle connections such that $\Gamma = \Gamma_{s+\log(2s^{\rho\epsilon})}(\omega)$, assume that it is contained in $\mathfrak{C}$ and let $m$ be the level of $\mathfrak{C}$. By Lemma \ref{lem:time big}  we have
\begin{equation*}
v(g_s\Gamma)\geq s^{-\frac 1 2+(\rho-4)\epsilon}.
\end{equation*}
By Lemma \ref{lem:sys to traj inside} there exists a vertical trajectory on $g_{s}\omega$ of length at least $\frac{1}{2D}s^{-\frac 1 2 +(\rho-4)\epsilon}$ which stays in $\mathfrak{C}$. Indeed, we have that
$$\frac 1 D v(\Gamma_s)-2v(\partial \mathfrak{C})\geq \frac 1 D s^{-\frac 1 2 +(\rho-4)\epsilon}-2\mathcal{M}s^{-\frac 12 +C\epsilon}\geq \frac 1 D s^{-\frac 1 2 +(\rho-4)\epsilon}-\frac{2}{8D}s^{(\rho-C-4)\epsilon}s^{-\frac 1 2 +C\epsilon},$$
where the last inequality uses the second condition of $p^*$.
This gives a vertical trajectory on $g_{s+\log(ms^{(C+3)\epsilon})}\omega$ of length at least $\frac 1 {2m}s^{-\frac 1 2 +(\rho-4-C-3)\epsilon}$.
By Lemmas \ref{lem:area} and \ref{lem:hor} on $g_{s+\log(ms^{(C+3)\epsilon})}\omega$ we have that $\frac Ah<m^2s^{-\frac 1 2 +(2C+3)\epsilon}$. So there exists a vertical trajectory of length at least 
\begin{equation}\label{eq:together}
\frac{ s^{(\rho - (3C+10))\epsilon }}{2 m^{3}2 D}\frac A h .
\end{equation}
We can apply Proposition \ref{prop:baby} with $N = \frac{s^{(\rho - (3C+10))\epsilon}}{2m^{3} D} = \frac{s^{(2+C)\epsilon/c}}{2m^{3} D}$, which yields a geodesic that can be shrunk to size at most $\sqrt{2} d^{3/2} N^{-c}\sqrt{A}$. Lemmas \ref{lem:area} and \ref{lem:hor}, respectively, yield an upper bound for the area $A$ and lower bound for the horizontal component $h$. Therefore, Proposition \ref{prop:baby} yields a geodesic which can be shrunk to size at most

\begin{equation}
\label{eqn:shrinking}
\begin{split}
\sqrt{2} d^{3/2} N^{-c}\sqrt{A} &\leq \sqrt{2} d^{3/2} N^{-c} m s^{-\frac{1}{2} + C\epsilon} = \sqrt{2}d^{3/2} \left(  \frac{s^{( 2+C )\epsilon/c}}{2m^{3}D}  \right)^{-c} m s^{-\frac{1}{2} + C\epsilon} \\
&= 2^{\frac{1}{2}+c} d^{3/2} D^{c}m^{3c+1} s^{-\frac{1}{2} - 2\epsilon} \leq 2^{\frac{1}{2}+c} d^{3/2} D^{c}\mathcal{M}^{3c+1} s^{-\frac{1}{2} - 2\epsilon} \\
&< (s+\log(s^{2\rho\epsilon}))^{-\frac{1}{2} - \epsilon}
\end{split}
\end{equation}

in time less than

\begin{equation}\label{eq:time}
\begin{split}
\frac{1}{2}\log\left(\frac{dN^2A}{h^2}\right) &= \log(N)+\log{h^{-1}} + \frac{1}{2}(\log(A) + \log(d)) \\
&\leq \log\left(\frac{s^{(2+C)\epsilon/c}}{2m^{3}D}\right) + \log(s^{\frac{1}{2} +3\epsilon}) + \frac{1}{2}\left( \log(m^{2} s^{-1+2C\epsilon}) + \log(d) \right)\\
&= \log\left( \frac{d^{\frac{1}{2}}s^{( 3+C+ (2+C)/c )\epsilon}}{2m^{2}D}  \right) \leq \log\left( s^{\rho\epsilon} \right).
\end{split}
\end{equation}
This contradicts that $\delta_t(\omega)>t^{-\frac{1}{2}-\epsilon}$ for $t>t_0$ and the result follows.

\section{Applying the setup}
\begin{prop}
\label{prop:contr}
There exists $\epsilon_0>0$ (which depends only on topology) and a $t_*>0$ (which depends on $\epsilon_0$) so that if $ \delta_t(\omega)>\frac 1 {t^{\frac 1 2 +\epsilon}}$ for some $\epsilon\in(0,\epsilon_0)$ and all $t>t_*$, then for any $t>t_*$ there exists $u\in [t,t+\log(t^{\mathcal{M}})]$ so that $\delta_u(\omega)>\frac 1 {u^{\frac 1 2 -\epsilon }}$, where $\mathcal{M}$ is the topological constant from Remark \ref{rmk:mbound}.
\end{prop}
Before proving the proposition we need a preliminary result which will let us apply Proposition \ref{prop:induct}. Given a flat surface let $\varphi^\theta_t$ be the flow in direction $\theta$ on it.

\begin{prop}
\label{prop:homotop nontrivial}Let $(X,\omega)$ be a flat surface of genus $g$. There exists $e>0$
 so that if  $\mathbf{C}$ is a subset of the triangles in a triangulation of $\omega$ and the perimeter of $C$ is at most $e$ then there exists a homotopically nontrival curve contained in $\mathbf{C}^c$. 
\end{prop} 

\begin{lem}
\label{lem:NonHom}
Let $(X,\omega)$ be a flat surface of genus $g$ and area 1 that is triangulated. There exists $\epsilon>0$ so that if 
 $\mathbf{C}$ 
 be a subset of the triangles and the perimeter of $\mathbf{C}$ is at most $\epsilon$ then there exists $p\in X$, $L,T, \in \mathbb{R}$, $\theta_1,\theta_2 \in S^1$ so that 
\begin{enumerate}
\item $p \in \mathbf{C}^c$
\item $(L,\theta_1)\neq (T,\theta_2)$
\item $\varphi_T^{\theta_1}(p)=\varphi_L^{\theta_2}(p)$ for some $T,L\neq 0$
\item $\varphi_\ell^{\theta_1}\cap \mathbf{C}=\varnothing$ for $0\leq \ell \leq T$ and $\varphi_\ell^{\theta_2}\cap \mathbf{C}=\varnothing$ for $0\leq \ell \leq L$.
\end{enumerate}
\end{lem}
\begin{proof}
  Choose $\epsilon<\frac 1 {20000}$ and let $\mathbf{C}\subset X$ have perimeter less than $\epsilon$. We claim 
that for each $\theta$, 
\begin{equation}\mu_\omega(\{p \notin \mathbf{C}: \varphi_\ell^{\theta}(p) \cap \mathbf{C} \neq \varnothing \text{ for some }0\leq \ell \leq 100\})<\frac 1 {100}.
\end{equation}
 Indeed, 
by our perimeter condition the area of $\mathbf{C}$  is at most $\epsilon^2$. For each $\theta$ the measure of $p$ so that the flow in direction $\theta$ between $0$ and $T$ crosses the boundary of the complex is at most $\epsilon T$ establishing the equation. Thus for every $\theta$ we have 
$$\mu_\omega(\{p\in \mathbf{C}^c: \varphi^{\theta}_t(p)\cap \mathbf{C}=\varnothing \text{ for all }0\leq \ell \leq 100\})>.99.$$
So by Fubini's theorem we have that there exists $p\in \mathbf{C}^c$
$$\lambda(\{\theta: \varphi^\theta_\ell(p)\cap \mathbf{C}=\varnothing \text{ for all }0\leq \ell\leq 100\})>\frac 1 2.$$ Call the set of angles $G_p$. 
 Now  $\int_0^{100}\int_{G_p}\chi_{\mathbf{C}^c}(\varphi_t^\theta(p))d\theta dt=100\lambda(G_p)> \lambda^2(\mathbf{C}^c)$. This implies that there exists $x \in \mathbf{C}^c$, $(\theta_1,T)\neq (\theta_2,L)\in G_p\times (0,100)$ with $\varphi_{T}^{\theta_1} (p)=x=\varphi_L^{\theta_2} (p)$.
\end{proof}

\begin{proof}[Proof of Proposition \ref{prop:homotop nontrivial}]
By Lemma \ref{lem:NonHom}, there exists $q \in \mathbf{C}^c$, $(\theta_1,T)\neq (\theta_2,L)\in G_p\times (0,100)$ with $\varphi_{T}^{\theta_1} (p)=q=\varphi_L^{\theta_2} (p)$.
If $\theta_1= \theta_2$ then since $T\neq L$ there exists a closed geodesic in direction $\theta_1$ from $q$ to $q$ and therefore there is a curve which is not homotopically trivial. So now we assume $\theta_1\neq \theta_2$. We claim that
$$\gamma=\bigcup_{\ell\in [0,T]}\varphi_\ell^{\theta_1}(p)\cup \bigcup_{[\ell \in [0,L]}\varphi_{-\ell}^{\theta_2}(q)$$
is a homotopically non-trivial curve. Indeed, $$\int_{\gamma}\omega=T\cos(\theta_1)+iT\sin(\theta_1)-L\cos(\theta_2)-iL\sin(\theta_2)$$ which is non-zero because $\theta_1\neq \theta_2$ and so $\gamma$ is not homotopically trivial.
\end{proof}
The next proof is involved. The idea is to iteratively apply Proposition \ref{prop:induct} to triangulate larger and larger subsets of the surface by short saddle connections, eventually triangulating the entire surface. Each time we apply it, the constants to plug into the statement of Proposition \ref{prop:induct} get worse. See Equation (\ref{eqn:timeBND}) and the line following it. We need to apply it the number of times equal to the number of saddle connections in a triangulation of the surface. With this in mind, we begin the proposition choosing constants so that we will arrive at contradictions in Equations (\ref{eqn:sTimes}) and (\ref{eq:contradiction}). 
\begin{proof}[Proof of Proposition \ref{prop:contr}] We prove this by contradiction. Pick $\alpha > 1$ such that $\alpha x  \geq 2 + 2((2+x)/c + 3x + 10)$ for all $x\geq 1$, where $c$ is the constant from Proposition \ref{prop:baby}, and set $\epsilon_0$ so that
\begin{equation}\label{eq:new}\epsilon_0 \leq \frac{1}{4} \alpha^{-\mathcal{M}}
\end{equation}
where $\mathcal{M}$ is the topological constant from Remark \ref{rmk:mbound}, and $\mathcal{M}\geq 2^{\mathcal{M}}\epsilon_0$. We now pick some $\epsilon\in(0,\epsilon_0)$.

Let $t_*$ be large enough so that
\begin{enumerate}[label=(\Roman*)]
\item it satisfies Proposition \ref{prop:induct} using $\epsilon$ as above and $C=1$;
\item \label{tcond:tbig} we have that $t> 3C^*_\mathcal{M}\epsilon \log (t)$ for all $t>t_*$, where $C^*_\mathcal{M}$ is defined below and only depends on the topology of the surface;
\item \label{item:less e} we have $\mathcal{M}t^{-\frac 1 2 +C_\mathcal{M}^*\epsilon}_*<e$, where $e$ is the constant from Proposition \ref{prop:homotop nontrivial};
\item $2^\mathcal{M} \leq t_*$;
\item For any $s>t_*$ and for all $\lambda\in [5,C^*_\mathcal{M}]$, we have $6s^{-\frac{1}{2} + (\lambda-1)\epsilon}\leq (s + \log(s^{\lambda-2}))^{-\frac{1}{2} + \lambda\epsilon}$, where $C^*_\mathcal{M}$ is defined below and depends only on the topology of the surface;
\item \label{item:area}$\mathcal{M}^2t_*^{-\frac{1}{2}} < 1$.
\end{enumerate}
From now on we assume that $t_0 \geq t_*$.

Let $\Gamma_{t_0}$ be a collection of saddle connections so that $|\Gamma_{t_0}| = \delta_{t_0}(\omega)$, let $\mathfrak{C}_m$ be a complex such that $\partial \mathfrak{C}_m = \Gamma_{t_0}$, and let $m$ be the number of saddle connections of $\Gamma_{t_0}$. Setting $C_m = 1$, we define recursively
$$\rho_m = \frac{2+C_m}c  + 3C_m + 10 \hspace{.5in}\mbox{ and }\hspace{.5in} C_{m+1} = 2+2\rho_m$$
to obtain $\{\rho_m,\dots, \rho_\mathcal{M}\}$ and $\{C_m,\dots, C_\mathcal{M}\}$. Note that by our choice of $\alpha$ and the way we recursively defined $C_k$ from $C_{k-1}$, we have that $C_k \leq \alpha C_{k-1}$ and that, therefore, 
\begin{equation}
\label{eqn:alpha}
 C_\mathcal{M}\leq \alpha^\mathcal{M} \leq (4\epsilon_0)^{-1}.
\end{equation}
Note that there is a universal $C^*_\mathcal{M}$ so that $C_k\leq C_{\mathcal{M}}^*$ for all $k\in\{m,\dots, \mathcal{M}\}$. Indeed, starting the recursive procedure with $m=1$, $C_1 = 1$, and recursively defining $C_2< \cdots<  C_\mathcal{M}$ as above, we get that $C_\mathcal{M} = C^*_\mathcal{M}$. 

Set $s_m = t_0$. At step $m\in\mathbb{N}$, we have a complex $\mathfrak{C}_m$ on $g_{s_m}\omega$ built from saddle connections bounded above by $s_m^{-\frac{1}{2} + C_m\epsilon}$. 
 By \ref{item:less e}, the perimeter of this complex is at most $ms_m^{-\frac 1 2 +C_m\epsilon}<\mathcal{M}s^{-\frac 1 2 +C_{\mathcal{M}}\epsilon}<e$, so by Proposition \ref{prop:homotop nontrivial} we satisfy the assumption that $\mathfrak{C}_m^c$ is not homotopically trivial in Proposition \ref{prop:induct}.
 By our assumption on the smallness of $\epsilon_0$ we have the assumption that $\delta_\ell(s)<\ell^{-\frac 1 2 +\epsilon}$ for all $\ell\in [s_m,s_m+\log(s_m^{2\rho_m\epsilon})]$ from Proposition \ref{prop:induct}. 
By Proposition \ref{prop:induct}  we have $\Gamma_{s_m + \log(s_m^{2\rho_m\epsilon})}(\omega)\not\subset \mathfrak{C}_m$.
We claim that
$$|g_{s_m}\Gamma_{s_m+\log(s_m^{2\rho_m\epsilon})}| \leq s_m^{-\frac{1}{2} + (1+2\rho_m)\epsilon}.$$
Indeed, this is because
$$|g_{s_m+\log(s_m^{2\rho_m\epsilon})}\Gamma_{s_m+\log(s_m^{2\rho_m\epsilon})}|\leq (s_m+\log(s_m^{2\rho_m\epsilon}))^{-\frac1 2 +\epsilon}$$
by our assumption on the length of systoles, and so
$$|g_{s_m}\Gamma_{s_m+\log(s_m^{2\rho_m \epsilon})}|\leq e^{\log(s_m^{2\rho_m\epsilon}) }|\Gamma_{s_m+\log(s_m^{2\rho_m\epsilon})}|\leq s_m^{2\rho_m\epsilon}(s_m+\log(s_m^{2\rho_m\epsilon}))^{-\frac1 2 +\epsilon}.$$
Let $s_{m+1} = s_m + \log(s_m^{2\rho_m\epsilon})$. By adding a saddle connection to $\mathfrak{C}_m$ through Lemma \ref{lem:combine} we obtain a new complex $\frak{C}_{m+1}$ on $g_{s_{m+1}}\omega$ whose saddle connections have lengths bounded above by $6s_m^{-\frac{1}{2} + (1+2\rho_m)\epsilon}$.  By the assumption on largeness of $t_0$ in (V), we have that
$$6s_m^{-\frac{1}{2} + (1+2\rho_m)\epsilon}\leq (s_m + \log(s_m^{2\rho_m}))^{-\frac{1}{2} + (2+2\rho_m)\epsilon}$$
and thus that
\begin{equation}
\label{eqn:timeBND}
6s_m^{-\frac{1}{2} + (1+2\rho_m)\epsilon}\leq (s_m + \log(s_m^{2\rho_m}))^{-\frac{1}{2} + (2+2\rho_m)\epsilon} \leq s_{m+1}^{-\frac{1}{2} + (2+2\rho_m)\epsilon},
\end{equation}
i.e., so that $\mathfrak{C}_{m+1}$ is a $s_{m+1}^{-\frac{1}{2} + C_{m+1} \epsilon}$-complex.

By iterating the above recursive procedure $\mathcal{M} - m$ times, we create a sequence of complexes $\mathfrak{C}_m\subset \mathfrak{C}_{m+1}$ each of which has that 
  perimeter less than $e$ (by \ref{item:less e} and the fact that $j\leq \mathcal{M}$ and $C_j\leq C^*_{\mathcal{M}}$) ending with $\mathfrak{C}_\mathcal{M} = X$, which is a $s_\mathcal{M}^{-\frac{1}{2} + C_\mathcal{M}\epsilon}$-complex. For any $k\in\{m+1,\dots, \mathcal{M}\}$, by assumption (II) on largeness of $t_0$ and since $C_{\mathcal{M}}>\rho_j$ for all $j<\mathcal{M}$, we can bound
\begin{equation}
\label{eqn:sTimes}
\begin{split}
s_k &\leq (s_{k-1}+\log(s_{k-1}^{2\rho_{k-1}\epsilon})) \leq 2s_{k-1} \leq 2(s_{k-2} + \log(s_{k-2}^{2\rho_{k-2}\epsilon})) \\
&\leq 2\cdot2s_{k-2} \leq 2\cdot 2(s_{k-3}+\log(s_{k-3}^{2\rho_{k-3}\epsilon}))\leq \cdots \leq 2^\mathcal{M}s_m
\end{split}
\end{equation}
for any $k\in\{m,\dots, \mathcal{M}\}$. The procedure of adding saddle connections to construct larger and larger complexes until obtaining $\mathfrak{C}_\mathcal{M} = X$ was done in time
\begin{equation} \label{eq:contradiction}
\begin{split}
\sum_{k=m}^\mathcal{M} r_k &\leq \sum_{k=m}^\mathcal{M} \log (s_k^{2\rho_k\epsilon}) \underset{{\footnotesize \mbox{by } (\ref{eqn:sTimes})}}{\leq} \sum_{k=m}^\mathcal{M} \log ((2^\mathcal{M}s_m)^{2\rho_k\epsilon}) \underset{{\footnotesize \mbox{by (IV)}}}{\leq} \sum_{k=m}^\mathcal{M} \log ( (s_m^2)^{2\rho_k\epsilon}) \\
&\leq \sum_{k=m}^\mathcal{M} \log ( (s_m^2)^{C^*_\mathcal{M}\epsilon}) \leq \mathcal{M} \log (s_m^{2C^*_\mathcal{M}\epsilon}) \leq \mathcal{M}\log (s_m),
\end{split}
\end{equation}
where the last inequality follows from (\ref{eq:new}) and (\ref{eqn:alpha}). In other words, we see that this procedure was done in the interval $[s_m,s_m+\log(s_m^{\mathcal{M}})]$. Using the estimate (\ref{eqn:alpha}), since ${\mathfrak{C}_\mathcal{M} = X}$ is a $s_\mathcal{M}^{-\frac{1}{2} + C_\mathcal{M}\epsilon}$-complex, by Lemma \ref{lem:area}, the area is bounded above by ${\mathcal{M}^2s_\mathcal{M}^{-1 + 2C_\mathcal{M}\epsilon} \leq \mathcal{M}^2s_\mathcal{M}^{-\frac{1}{2}} < 1}$ (by \ref{item:area}). However, this contradicts that the area of $(X,\omega)$ is 1, and the result follows.
\end{proof}
\section{Unique ergodicity}
\label{sec:ue}
\begin{prop}
\label{prop:gap}
Let $(X,\omega)$ be a flat surface. Suppose there is a $ \frac 1 2 > c > 0$ and a set of positive upper density $\mathcal{S}$ such that 
$$\frac{-\log \delta_t(\omega)}{\log\, t} \leq \frac{1}{2} - c$$
for $t\in \mathcal{S}$. Then the vertical flow on $(X,\omega)$ is uniquely ergodic.
\end{prop}

\begin{proof}
By definition we have that
$$\bar{d}(\mathcal{S}) \equiv \limsup_{s\in\mathcal{S}}\frac{|\mathcal{S}\cap [0,s]|}{s} > 0.$$
Let $s_k\rightarrow\infty$ be a sequence of times such that 
$$\bar{d}(\mathcal{S})-\frac{|\mathcal{S}\cap[0,s_k]|}{s_k}\leq \frac{1}{k}.$$

Let $T = \max\{1,\inf \mathcal{S}\}$ and $K>0$ be such that $s_k(1-(\bar{d}(\mathcal{S}) - \frac{1}{k})) > T$ for all $k>K$. Since $t^{c -\frac{1}{2}}$ is a decreasing function,
\begin{equation*}
\label{eqn:sysBnd}
\begin{split}
\int_0^\infty \delta_t^2(\omega)\, dt &\geq \int_\mathcal{S} \delta_t^2(\omega)\, dt \geq \int_\mathcal{S} t^{-1 + 2c}\, dt \geq \int_{\mathcal{S}\cap[T, s_k]}t^{-1 + 2c}\, dt \\
&\geq \int_{s_k(1-(\bar{d}(\mathcal{S}) - \frac{1}{k}))}^{s_k} t^{-1+2c}\, dt= \frac{s_k^{2c}}{2c}(1-(1-(\bar{d}(\mathcal{S}) - k^{-1}))^{2c})
\end{split}
\end{equation*}
for all $k>K$. Since $s_k\rightarrow\infty$, we have that $\int_0^\infty\delta_t^2(\omega)\, dt = \infty$ and therefore, by Theorem \ref{thm:UE}, the vertical flow on $(X,\omega)$ is uniquely ergodic.
\end{proof}
\begin{lem}
\label{lem:density}
Let $(X,\omega)$ be a flat surface that satisfies the logarithmic law (\ref{eqn:SystLogLaw}). There exists a $\lambda_0\in (0,\frac{1}{2})$ such that the set
$$\mathcal{S}_\lambda(\omega) := \left\{ s\in \mathbb{R}^+ : \delta_s(\omega) > s^{-\frac{1}{2} + \lambda}\right\}$$
has positive lower density for any $\lambda\in(0,\lambda_0)$.
\end{lem}
\begin{proof} By Proposition \ref{prop:contr} for all $\epsilon$ small enough we have that for all large enough $t_0$ there exists $u\in J_0 := [t_0,t_0+\log(t_0^\mathcal{M})]$ so that $\delta_u(\omega)>\frac 1 {u^{\frac 1 2 -\epsilon}}$. It follows that on the intervals $(u,u+\frac {\epsilon}4\log t_0^\mathcal{M})$ and $(u-\frac{\epsilon}4\log(t_0^\mathcal{M}),u)$ 
 we have that $\delta_s(\omega) > \frac 1 {s^{\frac 1 2 -\frac {\epsilon} 4}}$ for all $s$ in this interval of length at least $\frac{\epsilon}{4}\log (t_0^\mathcal{M})$. One of these is contained in $[t_0,t_0+\log(t_0^\mathcal{M})]$.  Therefore, the fraction of time in $J_0$ that $\delta_s(\omega) > \frac 1 {s^{\frac 1 2 -\frac {\epsilon} 4}}$ is at least $\frac{\epsilon}{4}$. Letting $t_k = t_{k-1} + \log(t_{k-1}^\mathcal{M})$ and considering $J_k = [t_k, t_k + \log(t_k^\mathcal{M})]$ we proceed in the same way to find an interval of at least $\epsilon/4$ fraction of $J_k$ with $\delta_s(\omega) > \frac 1 {s^{\frac 1 2 -\frac {\epsilon} 4}}$ in this interval. Since we cover $[t_0,\infty)$ in this way, the result follows.
\end{proof}
\begin{proof}[Proof of Main Theorem \ref{thm:main}]
Let $(X,\omega)$ satisfy the logarithmic law (\ref{eqn:SystLogLaw}). By Lemma \ref{lem:density}, for any $\varepsilon$ small enough, $\mathcal{S}_\varepsilon(\omega)$ has positive lower density. The result follows from Proposition \ref{prop:gap}.
\end{proof}
\section{Geometry and extremal length}
\label{sec:geometry}
We briefly review some results about the geometry of moduli space and we follow the survey \cite{MasurGeometrySurvey}. For a Riemann surface $X$, let $\mathcal{Q}$ be the set of homotopy classes of homotopically nontrivial essential simple closed curves on $X$.
\begin{defin}
The \textbf{extremal length} of a class $\alpha\in \mathcal{Q}$ on $X$ is the quantity
$$\mathrm{Ext}_X(\alpha) = \sup_\sigma \frac{L_\sigma^2(\alpha)}{A(\sigma)}$$
where the supremum is over the set of metrics $\sigma$ which are conformally-equivalent on $X$, $L_\sigma^2(\alpha)$ is the infimum of all lengths of curves which are in the homotopy class $\alpha$ measured with respect to the metric $\sigma$, and $A(\sigma )$ is the area of $X$ as measured by the metric $\sigma$.
\end{defin}
The following theorem, due to Kerckhoff \cite[Theorem 4]{Kerckhoff}, relates extremal length to distance in moduli space.
\begin{thm}
\label{thm:kerckhoff}
The Teichm\"uller distance between two surfaces $X,Y$ is given by
\begin{equation}
\label{eqn:kerckhoff}
dist(X,Y) = \sup_{\alpha\in \mathcal{Q} } \frac{1}{2} \log \frac{\mathrm{Ext}_X(\alpha)}{\mathrm{Ext}_Y(\alpha)}.
\end{equation}
\end{thm}
\subsection{Low genus}
In section \ref{sec:example} we will consider a genus two example. We review some background for flat surfaces of low genus which we will use in \S \ref{sec:example}. For context that will be useful in our construction, we recall the following result of McMullen \cite[Theorem 1.7]{McMullen}.
\begin{thm*}
\label{thm:mcmullen}
Let $(X,\omega)$ be a flat surface of genus two. Then it can be written, in infinitely many ways, as a connected sum $(X,\omega) = (E_1,\omega_1) \# (E_2,\omega_2)$ with $(E_i,\omega_i)\in\mathcal{A}_1$.
\end{thm*}
If $(X,\omega)$ is a genus 2 surface and $\gamma$ is a curve so that $(X,\omega)$ is a connect sum of two tori along $\gamma$ we call $\gamma$ a \emph{slit}. The slit defines homotopically non-trivial curve of zero cohomology class (it separates the surface). Let $\alpha_1,\beta_1,\alpha_2,\alpha_2$ be homotopy classes of curves such that $\langle  \alpha_1,\beta_1,\alpha_2,\alpha_2   \rangle = H_1(X,\mathbb{Z})$. Since $\pi_1 (X) = \langle \alpha_1,\beta_1,\alpha_2,\alpha_2 | [\alpha_1,\beta_1][\alpha_2,\beta_2] \rangle$, where $[a,b]$ is the commutator of $a$ and $b$, let $\sigma = [\alpha_1,\beta_1][\alpha_2,\beta_2]$, that is, the non-trivial homotopy class which separates the surface into two slitted tori, i.e., the class of the curve going around the slit. We can choose the $\alpha_i,\beta_i$ such that they form a symplectic basis.

For any surface $(X,\omega)\in\mathcal{A}_2^{(1)}$, we will denote by $\delta^s(\omega)$ the length of the shortest homotopically nontrivial \emph{separating} closed curve on $(X,\omega)$ and by $\delta^{\not s}(\omega)$ the length of the shortest homotopically nontrivial non-separating closed curve on $(X,\omega)$. Let $\delta_t^s(\omega) = \delta^s(g_t \omega)$ and $\delta_t^{\not s}(\omega) = \delta^{\not s}(g_t \omega)$ and note that $\delta_t(\omega) \in \{\delta_t^{\not s}(\omega),\delta_t^s(\omega)\}$, where we always measure lengths with respect to the flat metric on $g_t(X,\omega)$.

Note that if $\delta_t(\omega) = \delta_t^{\not s}(\omega)$ then flat geodesic of length $\delta_t(\omega)$ is a member of a cylinder $A$ foliated by trajectories parallel to this geodesic. The modulus of this cylinder, denoted $\mathrm{Mod}(A)$, is $h/\delta_t(\omega)$, where $h$ is the height of the cylinder. Note that in this case, by area considerations, we have that $\mathrm{Mod}(A)\leq     (\delta_t(\omega))^{-2}$.

In the case when $\delta_t(\omega) = \delta_t^{s}(\omega)$, the geodesic with length $\delta_t(\omega)$ is not contained in a foliated cylinder, as this would otherwise mean that the geodesic was non-separating since core curves of cylinders have non-trivial homology classes. In this case, the geodesic is either the boundary curve or the core curve of an annulus $A$, called an \emph{expanding annulus}. We will also need a bound for the conformal modulus of this annulus. 

In order to use Kerckhoff's formula (\ref{eqn:kerckhoff}) we will bound the extremal length of curves by the moduli of cylinders and annuli which contain them. The following can be gathered from \cite[\S 5]{CRS:lines} or \cite[Lemma 3.6]{R} and provides the crucial bounds for the extremal lengths in terms of moduli of annuli and cylinders.
\begin{thm}
Let $(X,\omega)$ be a flat surface of unit area and denote by $X_t$ the Riemann surface on which $g_t \omega$ is holomorphic. There exist constants $K_s,K_{\not s}>0$ which depend only on the topology of the surface such that
$$\frac{1}{\min_{\gamma\in \mathcal{Q}} \mathrm{Ext}_{X_t}(\gamma)} \leq \max\{  - K_s \log (\delta_t^s(\omega))+K_s ,  K_{\not s}\cdot(\delta_t^{\not s})^{-2}+K_{\not s} \}.$$
\end{thm}
This gives a bound for Kerckhoff's distance (\ref{eqn:kerckhoff}).
\begin{cor}
\label{cor:KerckhoffDist}
Let $(X,\omega)\in \mathcal{A}_2^{(1)}$ and denote by $X_t$ the Riemann surface on which $g_t\omega$ is holomorphic. Then there exists a constant $K_2$ which depends only on the topology of the surface such that
$$\mbox{dist}(X,X_t) \leq   \max\left\{ \frac{1}{2}\log( -\log (\delta^s_t(\omega))), - \log (\delta_t^{\not s}(\omega))\right\} + K_2.$$
\end{cor}

\section{Logarithm laws and non-ergodicity in genus 2}
\label{sec:example}
In this section we construct the example which yields Main Theorem \ref{thm:negative}. Recall that if $\alpha \in(0,1)\setminus \mathbb{Q}$, it has a unique continued fraction expansion $\alpha=[a_1,...]$. Its best convergents $\frac{p_k}{q_k}$ are given by $p_0=0$, $p_1=1=q_0$ and inductively $q_{k+1}=a_{k+1}q_k+q_{k-1}$ and $p_{k+1}=a_{k+1}p_k+p_{k-1}$.
Let $\alpha\in [0,1)$ have continued fraction expansion $a_i=\lceil (i+10) (\log(i+1))^2\rceil $ for all $i$. Let $\|n\alpha\|=d(n\alpha,\mathbb{Z})$ and $R_\alpha(x)=x+\alpha-\lfloor x+\alpha \rfloor$ the rotation by $\alpha$. Further recall,
\begin{lem}\label{lem:sep}$\underset{0< i\leq q_{k}-1}{\min}\|i\alpha\|=\|q_{k-1}\alpha\|.$
\end{lem}

Define the lattice given by $\alpha$
$$\Lambda_\alpha := \begin{pmatrix} 1&-\alpha\\0&1\end{pmatrix}\mathbb{Z}^2$$
and let $(E_i,\omega) = (\mathbb{C}/\Lambda_\alpha,dz)$, $i = 1,2$, be two copies of the same flat torus given by the lattice $\Lambda_\alpha$. We will denote by 
$$(\hat{X},\hat{\omega}) = (E_1,dz)\#_\alpha (E_2,dz)$$ 
the genus two translation surface formed by gluing $(E_1,dz)$ to $(E_2,dz)$ along a slit of holonomy $(\sum_{k=1}^\infty 2\|q_{k}\alpha\|,0)$. The next lemma implies that $\sum_{i=1}^\infty 2\|q_i\alpha\|<1$.
\begin{lem} \label{lemma:good:bound}$\frac 1 {q_{k+1}+q_k}<\|q_{k}\alpha\|<\frac 1 {a_{k+1}q_k}.$
\end{lem}
See \cite{khinchin}, 4 lines before equation 34.

\begin{prop} Consider the following two sets: 
\begin{itemize}
\item The set of period curves of cylinders with holonomy either $(\pm \|q_{k}\alpha\|,q_{k})$ or $(\pm2 \|q_{k}\alpha\|,2q_{k})$ for some $k$. There exists a constant $C$ so that the shortest these curves get, under the Teichm\"uller deformation, is between $\frac 1 C \sqrt{\frac 1 {a_{k+1}+2}}$ and $C\sqrt{\frac 1 {a_{k+1}+2}}$. Their lengths on $g_t\hat{\omega}$ is at least 1 if $t<\log(q_{k})$.
\item The set of slits with holonomy $(\sum_{i=k}^{\infty}2\|q_{i}\alpha\|,\sum_{i=1}^{k-1}2q_{i})$ for some $k$. The shortest such a  curve gets is $\sqrt{q_{k-1}\sum_{i=k}^{\infty}2\|q_{i}\alpha\|}$ which is at least $2\sqrt{\frac {q_{k-1}} {q_{k+1}}}$ for all $t$. Its length on $g_t\hat{\omega}$ is at least $1$ if $t<\log(q_{k-1})$.
\end{itemize}
For all $t$ we have that there exists $\gamma_t$ in the first set and $\zeta_t$ in the second set so that $\delta_t^s<\frac 1 9g_t\zeta_t$ and $\delta^{\not s}_t<\frac 1 9 g_t\gamma_t$. 
\end{prop}
\begin{proof}
Every saddle connection on $(\hat{X} ,\hat{\omega})$ has the form of a vector connecting points in $\Lambda_\alpha \cup \left( (\sum_{i=1}^{\infty}2\|q_{i}\alpha\|,0)+\Lambda_\alpha \right)$. 

We first wish to consider primitive vectors in this set. We may assume the vector starts from $(0,0)$, has vertical component of its holonomy $n$ for some $n \in \mathbb{N}$. By our assumption that our simple closed curve gets small under $g_t$ we further restrict our attention to when such a saddle connection has that the horizontal component of its holonomy is less than $\frac{2c^2}{n}$. 
There are two possibilities: the vector connects $(0,0)$ to $(x,y)\in \Lambda_\alpha$ or
 it connects $(0,0)$ to $(x,y) \in\Lambda_\alpha +(\sum_{i=1}^{\infty}2\|q_{i}\alpha\|,0)$. We consider the first option. Observe that the closest element of $\Lambda_\alpha$  to $(0,n)$ is $(\pm \|n\alpha\|,n)$ and so the horizontal holonomy of the shortest curve with vertical holonomy $n$ has horizontal holonomy $d(R^n_\alpha(0),0)$. 
  
\textbf{Sublemma:} For all $k,j,p$ there exists at most 1 element of $\{j+i\}_{i=1}^{q_k}$ so that $d(R^{j+i}_\alpha(0),p)<\frac 1 {8q_k}$. 

\begin{proof}By Lemma \ref{lem:sep} we have $\{R^{i+j}_\alpha(0)\}_{i=1}^{q_k}$ is $\|q_{k-1}\alpha\|$ separated. By Lemma \ref{lemma:good:bound} this is at least $\frac 1 {4q_k}$. If a set of points are $\frac 1 {4q_k}$ separated, at most one element of the set can be within $\frac 1 {8q_k}$ of a given point.
\end{proof}

It is straightforward that $\|bq_k\alpha\| <  \|i\alpha\|$ for all $0<i<q_{k+1}$ so that $i \notin \{\ell q_k\}_{\ell=1}^{a_{k+1}}$ and $b\leq a_{k+1}$. So by the sublemma we have that any saddle connection that can be made short has $n=b q_k$ for some $k$ and $b<a_{k+1}$ (indeed any two $j$ so that $d(R^j_\alpha(0),0)<\frac{c}{j}$ are at least $q_{{k}}$ separated if $c<\frac 1 4$). These give closed curves on the torus, so either this curve is a closed curve on $\hat{\omega}$ or twice it is a closed curve on $\hat{\omega}$. It is obvious that
$$\sqrt{(b \|q_k\alpha\|)^2+(bq_k)^2}>\frac 1 9 \sqrt{(2\|q_k\alpha\|)^2+(2q_k)^2}$$
for all $1\leq b\leq a_{k+1}$. 

We now consider the other case. First observe that for all $k$ we have
$$d\left(\sum_{i=1}^\infty2 \|q_i\alpha\|,R^{\sum_{i=1}^k 2q_i}_\alpha(0)\right)=\sum_{i=k+1}^\infty 2 \|q_i\alpha\|.$$ 
Observe that there is a saddle connection with holonomy  $(\sum_{i\geq k}2\|q_{i}\alpha\|,\sum_{j=1}^{k-1}2 q_{j})$. Notice that it connects two different cone points. Consider the union of this saddle connection and its image under the involution $+1$ that interchanges the two tori. The union of these two gives a simple closed geodesic that disconnects the surface. See for example  \cite[Section 2.3]{CMW:limits}. Call this curve $\zeta_k$. 

Now if $\xi$ is a saddle connection connecting $(0,0)$ to $\Lambda_\alpha +(\sum_{i=1}^{\infty}2\|q_{i}\alpha\|,0)$ with vertical holonomy between $q_k,q_{k+1}$ we have that for all $t$, $g_t\xi>\frac 1 9 g_t\zeta_k$. To see this first notice that because the vertical holonomy of $\zeta_k<9q_k$ if the holonomy of $g_t\xi\leq \frac 1 9g_t \zeta_k$ then the horizontal holonomy of $\xi$ must be smaller than $\zeta_k$'s. It is straightforward to check that no saddle connections we are considering (other than the two equal length saddle connections that make up $\zeta_k$) have this property. 
\end{proof}

\begin{cor}
\label{cor:loglaw}
Let $(\hat{X},\hat{\omega}) = (\mathbb{C}/\Lambda_\alpha,dz)\#_\alpha (\mathbb{C}/\Lambda_\alpha, dz)$ be the flat surface of genus two constructed as above. For our choice of $\alpha$ we have that 
$$\limsup_{t \to \infty} \frac{\max\left\{\frac{1}{2}\log(-\log(\delta^s_t(\hat{\omega}))),-\log(\delta_t^{\not s}(\hat{\omega}))  \right\}}{\log t}\leq \frac 1 2.$$
\end{cor}
\begin{proof} By the previous proposition, it suffices to show that 
$$\limsup_{k \to \infty} \frac{\max\left\{\log(\sqrt{a_{k}}),\log(\log(\frac{q_{k+1}}{q_{k-1}}))\right\}}{\log \log q_{n_{k-1}}} \leq \frac{1}{2}.$$
  Since
  \begin{equation*}
\begin{split}
  a_k&=\left\lceil (k+10) \log(k+1)^2\right\rceil, \\
  \frac{q_{k+1}}{q_{k-1}}&<10\left(\left\lceil (k+10)\log(k+1)^2(k+11)\log(k+2)^2\right\rceil\right),
  \end{split}
    \end{equation*}
and for all large enough $k$ we have $q_k>2^k$, the corollary follows.
\end{proof}

\begin{lem} 
\label{lem:nonerg}
The vertical flow on $(\hat{X},\hat{\omega})$ is not ergodic.
\end{lem}
\begin{proof}
By a result of Veech \cite[Theorem 3]{veech:strict} and of Keynes-Newton \cite[Lemma 4]{KeynesNewton}, it suffices to show that $\sum 2q_{n_k}\|q_{n_k}\|<\infty$. By Lemma \ref{lemma:good:bound} this is at most $\sum_{k=1}^\infty \frac 2 {a_{n_k+1}}<\infty$. This follows because we assume $a_{n_k+1}>k\log(k)^2$.
\end{proof}

\begin{proof}[Proof of Main Theorem \ref{thm:negative}]
By Lemma \ref{lem:nonerg}, the vertical flow on the genus two surface $(\hat{X},\hat{\omega})$ constructed by glueing two copies of the torus $(\mathbb{C}/\Lambda_\alpha, dz)$ along a slit of holonomy $(\sum_{k=1}^\infty 2\|q_{n_k}\alpha\|,0)$ is non-ergodic. By Corollaries \ref{cor:KerckhoffDist} and  \ref{cor:loglaw}, we have that the surface $(\hat{X},\hat{\omega})$ satisfies the logarithmic inequality (\ref{eqn:inequality}).
\end{proof}

\bibliographystyle{amsalpha}
\bibliography{biblio}

\end{document}